\documentclass[10pt]{amsart}
\usepackage[english]{babel}
\usepackage{amssymb}
\usepackage{enumitem}
\usepackage{hyperref}
\usepackage[initials]{amsrefs}
\usepackage[all]{xy}
\SelectTips{cm}{}

\usepackage{color}

\newcommand{\bbN}{{\mathbb N}}

\newcommand{\bfG}{{\mathbf G}}

\newcommand{\Aut}{\operatorname{Aut}}

\newcommand{\inn}{\operatorname{inn}}

\newcommand{\PGL}{\operatorname{PGL}}
\newcommand{\GL}{\operatorname{GL}}
\newcommand{\Gr}{\operatorname{Gr}}

\newcommand{\LL}{\operatorname{L}^0}

\newcommand{\dd }{\,{\rm d}}

\newcommand{\overto}[1]{{\buildrel{#1}\over\longrightarrow}}

\newcommand{\acts}{\curvearrowright}

\newtheorem{theorem}{Theorem}[section]
\newtheorem{lemma}[theorem]{Lemma}

\newtheorem*{claim}{Claim}

\newtheorem{cor}[theorem]{Corollary}
\newtheorem{proposition}[theorem]{Proposition}
\newtheorem{prop}[theorem]{Proposition}

\theoremstyle{definition}

\newtheorem{defn}[theorem]{Definition}

\newtheorem{remark}[theorem]{Remark}

\numberwithin{equation}{section}
\begin{document}

\title{Super-Rigidity and non-linearity for lattices in products}

\author{Uri Bader}
\address{Weizmann Institute, Rehovot}
\email{bader@weizmann.ac.il}
\thanks{U. Bader was supported in part by the ISF-Moked grant 2095/15 and the ERC grant 306706.}

\author{Alex Furman}
\address{University of Illinois at Chicago}
\email{furman@uic.edu}
\thanks{A. Furman was supported in part by the NSF grant DMS 1611765.}


\maketitle

\begin{abstract}
We prove a super-rigidity result for algebraic representations over complete fields of irreducible lattices 
in products of groups and lattices with dense commensurator groups.
We derive criteria for the non-linearity of such groups.
\end{abstract}

\section{Introduction}

In this work we consider products of locally compact second countable groups (hereafter {\em lcsc groups}) and their lattices.

\begin{defn}
	Let $T$ be a locally compact second countable (hereafter lcsc) group. 
	A discrete subgroup $\Gamma<T$ is a \emph{lattice} in $T$  
	if the coset space $T/\Gamma$ carries a finite $T$-invariant measure.
	
	Let $T=T_1\times T_2\times \cdots\times T_n$ be a product of $n$ lcsc groups 
	and denote by $\pi_i:T\to T_i$ the projections.
	A subgroup $\Gamma<T$ is said to be a \emph{lattice with dense projections} in $T$ if $\Gamma<T$ is a lattice, 
	and $\pi_i(\Gamma)$ is dense in $T_i$ for every $i\in\{1,2,\dots,n\}$.
\end{defn}

In the above definition $n$ is an arbitrary integer, but the case $n=1$ is redundant (it forces $T$ to be discrete and $\Gamma=T$),
thus we are mostly concerned with $n\geq 2$.
The following theorem extends the case of the classical Super-Rigidity Theorem of Margulis, 
\cite[Chapter VII,~Theorem~(5.6)]{margulis-book} 
that concerns irreducible lattices in a semisimple Lie group of higher rank.
The case of a lattice in a simple Lie group of higher rank is treated in \cite{BF:Margulis}.
The target group $\mathbf{G}(k)$ below is defined over 
a valued field $k$ that need not be a local field. 
The notion of being \emph{bounded} for subsets of such a group $\mathbf{G}(k)$ 
-- a \emph{group bornology} as defined by Bruhat--Tits \cite{Bruhat+Tits}*{(3.1.1)} -- 
is discussed in detail in \cite{BDL}*{\S 4}.

\begin{theorem}[Super-Rigidity for lattices in products] \label{thm:lattice}\hfill{}\\
Let $T=T_1\times T_2\times \cdots\times T_n$ be a product of lcsc groups,
and let $\Gamma<T$ be a lattice with dense projections.
Let $k$ be a field with an absolute value.
Assume that as a metric space $k$ is complete.
Let $G=\mathbf{G}(k)$ be the $k$-points of a connected, adjoint, $k$-simple algebraic group.
Let $\rho:\Gamma \to G$ be a homomorphism.
Assume $\rho(\Gamma)$ is Zariski dense and unbounded in $G$.

Then there exists a unique continuous homomorphism $\bar{\rho}:T \to G$ satisfying $\rho=\bar{\rho}|_\Gamma$.
Moreover, that homomorphism is of the form $\bar{\rho}=\bar{\rho}_i\circ \pi_i$ for some (necessarily unique) $i\in\{1,\ldots,n\}$
and some (necessarily unique) continuous homomorphism $\bar{\rho}_i:T_i\to G$.
\end{theorem}

The proof of Theorem~\ref{thm:lattice} will be given in \S\ref{section:mainthm}.
Similar Super-Rigidity theorems for lattices in products were given by Monod \cite{Monod-products} 
and Gelander--Karlsson--Margulis \cite{GKM} 
under the further assumption that $\Gamma<T$ is cocompact (or satisfies a non-trivial integrability condition)
and by Caprace--Monod \cite[Theorems 5.1, 5.6]{CM09} under a finite generation assumption on $\Gamma$
(note also that the first two references consider a wider class of possible targets).
No integrability or finite generation assumptions are forced on $\Gamma$ in Theorem~\ref{thm:lattice}.
This is a hint that the above theorem could be generalized to the setting of a cocycle,
which is indeed the case.
However, we choose not to elaborate on this point here.

A nice application of the theory of lattices in products is given by Caprace and Monod
in the following Lemma.

\begin{lemma} [{Caprace--Monod, \cite{CM09}*{Lemma~5.15}}]
Let $S$ be a locally compact group, $\Gamma<S$ a lattice
and $i:\Lambda\hookrightarrow S$ a countable dense subgroup containing and commensurating $\Gamma$.
Then there exists a locally compact group $S'$ and a homomorphism $\theta:\Lambda\to S'$ such that $\theta(\Gamma)$ is precompact in $S'$ and $i\times\theta(\Lambda)<S\times S'$ is a lattice with dense projections.
\end{lemma}

The group $S'$ appearing in the above lemma is the completion of $\Lambda$ with respect to the left uniform structure generated by conjugates of $\Gamma$.
Applying Theorem~\ref{thm:lattice} with $T_1=S$, $T_2=S'$ and $T=T_1\times T_2$ we readily get the following corollary.

\begin{cor} [Super-Rigidity for commensurators] \label{cor:com}\hfill{}\\
Let $T$ be a lcsc group,
and let $\Gamma<T$ be a lattice.
Let $\Lambda$ be a countable dense subgroup of $T$ containing and commensurating $\Gamma$.
Let $k$ be a field with an absolute value.
Assume that as a metric space $k$ is complete.
Let $G$ be the $k$-points of a connected, adjoint, $k$-simple algebraic group.
Let $\rho:\Lambda \to G$ be a homomorphism.
Assume $\rho(\Lambda)$ is Zariski dense and $\rho(\Gamma)$ is unbounded in $G$.

Then there exists a unique continuous homomorphism $\bar{\rho}:T\to G$
such that $\rho=\bar{\rho}|_\Lambda$.
\end{cor}

This corollary is a generalization of another well known Theorem of Margulis, \cite[Chapter VII,~Theorem~(5.4)]{margulis-book}.
Here too, similar theorems were proved by Monod \cite[Theorem~A.1]{Monod-products},
Gelander--Karlsson--Margulis \cite[Theorem~8.1]{GKM} and Caprace--Monod \cite[Theorems 5.17]{CM09} 
under further assumptions on $\Gamma$ and $\Lambda$.

Note that Theorem~\ref{thm:lattice} and Corollary~\ref{cor:com} are new even in the case where $T$
is a semisimple group over a local field, as $k$ is not assumed to be a local field.
The analogous Super-Rigidity Theorems of Margulis
are cornerstones in his celebrated proof of his
Arithmeticity Theorems~\cite[Chapter~IX~(1.9),~Theorems~(A) and (B)]{margulis-book}.
However, due to the standing assumption that the target group is defined over a local field, 
he is forced to assume in the aforementioned Theorems (A) and (B) %, and then to argue using completely different methods,  
that the lattice $\Gamma$ is finitely generated.
Our new versions of these theorems show that these finite generation issues could be circumvented.
We thank T.N. Venkataramana for pointing out this application to us.
We should note that these assumptions were later removed, in case (A) by Raghunathan, 
\cite{Raghunathan} and in case (B) by Lifschitz \cite{Lifschitz}.
However, using our Theorem~\ref{thm:lattice} in case (A) and Corollary~\ref{cor:com} in case (B) is easier and avoids 
the detailed discussion of the shape of the cusps in locally symmetric spaces of function fields taken 
in \cite{Raghunathan} and relied upon in \cite{Lifschitz}.
We will elaborate further on this point in a forthcoming paper.

Theorem~\ref{thm:lattice} and Corollary~\ref{cor:com} have striking applications in case the group $T$ is not linear.
We say that a group $\Gamma$ is {\em solvable-by-locally finite}
if there exists a normal solvable subgroup $S \lhd \Gamma$ such that $\Gamma/S$ is locally finite, 
i.e. it is a directed union of finite subgroups.
We say that topological group $H$ is \emph{amenable} if every continuous $H$-action on a compact space
has an invariant probability measure.

\begin{theorem} \label{thm:nonlinearity}\hfill{}\\
Let $T=T_1\times T_2\times \cdots\times T_n$ be a product of lcsc groups,
and let $\Gamma<T$ be a lattice with dense projections.
Assume that for every continuous homomorphism $\psi:T\to \GL_d(k)$, where $k$ is a complete field with absolute value 
and $d$ is an integer,
the closure of the image, $\overline{\psi(T)}$, is amenable.
Then for every integer $d$, field $K$ and linear representation $\phi:\Gamma \to \GL_d(K)$, the image $\phi(\Gamma)$ 
is solvable-by-locally finite.

Furthermore, if $\Gamma$ is assumed to be finitely generated, the class of fields considered 
in the target of $\psi$ can be taken to be the class of local fields
and then the image $\phi(\Gamma)$ is solvable by finite.
\end{theorem}

\begin{cor} \label{cor:nonlinearitycom}\hfill{}\\
Let $T$ be a lcsc group,
and let $\Gamma<T$ be a lattice.
Let $\Lambda$ be a countable dense subgroup of $T$ containing and commensurating $\Gamma$.
Assume that for every continuous homomorphism $\psi:T\to \GL_d(k)$, where $k$ is a complete field with an absolute 
value and $d$ is an integer,
the closure of the image, $\overline{\psi(T)}$, is amenable.
Then for every integer $d$, field $K$ and linear representation $\phi:\Lambda \to \GL_d(K)$, the image $\phi(\Gamma)$ 
is solvable-by-locally finite.

Furthermore, if $\Lambda$ is assumed to be finitely generated, the class of fields considered in 
the target of $\psi$ could be taken to be the class of local fields
and then the image $\phi(\Gamma)$ is solvable by finite.
\end{cor}

%Under stronger assumptions we get a stronger result.

\begin{cor} \label{cor:newamenable}\hfill{}\\
Assume given for every $i=1,\ldots,n$ a totally disconnected lcsc group $T_i$ in which for every closed non-trivial normal subgroup the quotient group is amenable.
Let $T=T_1\times T_2\times \cdots\times T_n$
and let $\Gamma<T$ be a lattice with dense projections such that for each $i$, 
$\pi_i|_\Gamma$ is injective (and dense in $T_i$).
Assume that $T_1$ is non-amenable and for every continuous homomorphism $\psi:T_1\to \GL_d(k)$, where $k$ is a complete field with an absolute value and $d$ is an integer,
the closure of the image, $\overline{\psi(T_1)}$, is amenable.
Then for every integer $d$, field $K$ and linear representation $\phi:\Gamma \to \GL_d(K)$, the image $\phi(\Gamma)$ is solvable-by-locally finite. 

Furthermore, if $\Gamma$ is assumed to be finitely generated, the class of fields considered in the target of $\psi$ could be taken to be the class of local fields
and then the image $\phi(\Gamma)$ is solvable by finite.
\end{cor}

The proofs of the above three theorems
will be given in \S\ref{sec:proffnl}.
In their proofs we will use new linearity criteria, namely 
Theorem~\ref{lem:nonlinearity} and Corollary~\ref{cor:nonlinearity},
which are of independent interest.
These  will be proven in \S\ref{section:nonlinearity}.
Note that, unlike Theorem~\ref{thm:lattice}, Theorem~\ref{thm:nonlinearity} and Corollary~\ref{cor:newamenable}
are not tautological in the case $n=1$.

The last three results could be further strengthened if we 
assume that $\Gamma$ has property (T), or that we are in a situation
corresponding to the setting of \cite[Theorem~1.1]{BS},
where one can deduce further that the image of $\Gamma$ is finite 
in any linear representation.
We refer our reader to \cite{Monod} where such a theorem
is proven, along with a very nice converse: under certain natural
assumptions on a finitely generated lattice in a product, the only way it could have a 
linear representation with an infinite image is that it is an arithmetic lattice
to begin with. 
For a similar conclusion (in characteristic $0$) with no finite generation assumption, see \cite{BFS}*{Theorem~4.5}.

\subsection{Acknowledgments and disclaimers}

Theorem~\ref{thm:lattice} appeared in our manuscript \cite{AREA},
which we do not intend to publish as, in retrospect, we find it hard to read.
While writing this paper we made an effort to improve the presentation and to keep things concise and simple.
The paper \cite{AREA} contains further results, regarding higher rank lattices and general cocycle super-rigidity, 
on which we intend to elaborate in two forthcoming papers.

We owe thanks to many friends, numerous to be listed here, for their continuous stimulation and interest in our work.
Special thanks go to T.N. Venkataramana, Pierre-Emmanuel Caprace, Tsachik Gelander and the anonymous referee.
We are also indebted to 
Bruno Duchesne and Jean L\'{e}cureux for their contribution to this project.

%%%%%%%%%%%%%

\section{Algebraic varieties as Polish spaces} \label{alg perlim}

In this section we fix a field $k$ with an absolute value $|\cdot|:k\to [0,\infty)$
as defined and discussed in \cite{valued}.
We assume that the absolute value is non-trivial, and that $(k,|\cdot|)$ is complete (as a metric space) and separable
(in the sense of having a countable dense subset).
% The reals $\mathbf{R}$ and the complex numbers $\mathbf{C}$ equipped with the usual absolute value,
% are the only fields with Archimedean valuations.
The theory of manifolds over such fields is developed in \cite[Part II, Chapter I]{serre}.

We also fix a $k$-algebraic group ${\bf G}$.
We will discuss the category of $k$-${\bf G}$-varieties.
A $k$-${\bf G}$-variety is a $k$-variety endowed with an algebraic action of ${\bf G}$ which is defined over $k$.
A morphism of such varieties is a $k$-morphism which commutes with the ${\bf G}$-action.
% We denote this category by $\mathcal{A}(k,{\bf G})$.
By a $k$-coset variety we mean a variety of the form
%which is $k$-isomorphic to the variety
${\bf G}/{\bf H}$ for some $k$-algebraic subgroup ${\bf H} < {\bf G}$
(see \cite[Theorem 6.8]{borel}).
% We denote the full subcategory of $\mathcal{A}(k,{\bf G})$ consisting of $k$-coset varieties by $\mathcal{A}_c(k,{\bf G})$.

Each $k$-${\bf G}$-variety gives rise to a topological space: $V={\bf V}(k)$ endowed with its $(k,|\cdot|)$-topology.
Topological notions, unless otherwise said, will always refer to this topology.
In particular $G={\bf G}(k)$ is a topological group.
It is locally compact if $k$ is a local field.

The norm $|\cdot|$ on $k$ allows one to define the associated \emph{bornology} (see \cite{Bruhat+Tits}*{(3.1.1)}) 
on $k$, on the affine spaces $k^N$, and on affine $k$-varieties such as $\mathbf{G}(k)$. 
We refer to \cite{BDL}*{\S 4} for a detailed discussion of the bornology on $\mathbf{G}(k)$ induced by $(k,|\cdot|)$.
In the special case where $k$ is a local field, being a bounded subset of $\mathbf{G}(k)$ is equivalent to being precompact in the topology on $\mathbf{G}(k)$ induced by $k$.

Recall that a topological space is called Polish if it is separable and completely metrizable.
For a good survey on the subject we recommend \cite{kechris}.
We mention that the class of Polish spaces is closed under countable disjoint unions and countable products.
A $G_\delta$ subset of a Polish spaces is Polish so, in particular, a locally closed subset of a Polish space is Polish.
A Hausdorff space which admits a finite open covering by Polish open sets is itself Polish.
Indeed,
such a space is clearly metrizable (e.g.\ by Smirnov metrization theorem)
so it is Polish by
Sierpinski's theorem \cite[Theorem 8.19]{kechris} which states that the image of an open map from a Polish space to a separable metrizable space is Polish.
Sierpinski's theorem also implies that for a Polish group $K$ and a closed subgroup $L$, the quotient topology on $K/L$ is Polish.
Effros' Lemma \cite[Lemma 2.5]{effros} says that the quotient topology on $K/L$ is the unique $K$-invariant Polish topology on this space.

\begin{proposition} \label{polishing}
The $k$-points of a $k$-variety form a Polish space.
In particular, $G=\mathbf{G}(k)$ is a Polish group.
If ${\bf V}$ is a $k$-${\bf G}$-variety
then the $G$-orbits in $V$ are locally closed.
For $v\in V$ the orbit ${\bf G}v$ is a $k$-subvariety of ${\bf V}$.
The stabilizer ${\bf H}<{\bf G}$
is defined over $k$ and the orbit map ${\bf G}/{\bf H}\to {\bf G}v$ is defined over $k$.
Denoting $H={\bf H}(k)$, the induced map $G/H \to Gv$ is a homeomorphism,
when $G/H$ is endowed with the quotient space topology and $Gv$ is endowed with the subspace topology.
\end{proposition}

\begin{proof}
Since $k$ is complete and separable it is Polish and so is the affine space $\mathbb{A}^n(k) \simeq k^n$.
The set of $k$-points of a $k$-affine variety is closed in the $k$-points of the affine space, hence it is a Polish subspace.
It follows that the set of $k$-points of any $k$-variety is a Polish space,
as this space is a Hausdorff space which admits a finite open covering by Polish open sets, namely the $k$-points of its $k$-affine charts.

The fact that the $G$-orbits in $V$ are locally closed is
proven in the appendix of \cite{b-z}.
Note that in \cite{b-z} the statement is claimed only for non-archimedean local fields, but the proof is actually correct for any field with
a complete non-trivial absolute value, which is the setting of \cite[Part II, Chapter III]{serre}
on which \cite{b-z} relies.

For $v\in V$ the orbit ${\bf G}v$ is a $k$-subvariety of ${\bf V}$ by \cite[Proposition 6.7]{borel}.
The stabilizer ${\bf H}<{\bf G}$
is defined over $k$ by \cite[Proposition 1.7]{borel}
and we get an orbit map which is defined over $k$ by \cite[Theorem 6.8]{borel}.
Clearly $H$ is the stabilizer of $v$ in $G$ and the orbit map restricts to a continuous map from $G/H$ onto $Gv$.
Since $Gv$ is a Polish subset of $V$, as it is locally closed, we conclude by Effros' Lemma that the latter map is a homeomorphism.
\end{proof}

%%%%%%%%%%%%%%%%

\section{Ergodic Theoretical preliminaries} \label{sec:AME}

In this section we review some notions and facts from Ergodic Theory.
Recall that a standard Borel space is a measurable space which is isomorphic as such to a Polish topological space endowed with its Borel $\sigma$-algebra
and a Lebesgue space is a standard Borel space endowed with a $\sigma$-finite measure class.

Given a Lebesgue space $Y$ and a Borel space $V$ we denote by $\LL(Y,V)$ the space of all classes of measurable maps, identified up to a.e equality, from $Y$ to $V$.
The elements of $\LL(Y,V)$ are called {\em Lebesgue maps} from $Y$ to $V$.
Given a standard Borel space $V$, a Lebesgue space $Y$ and an essentially surjective Borel map $\pi: V\to Y$
(that is, $\pi$ is measurable for a Borel model of $Y$ and its image has full measure),
we denote by $\LL(\pi)$ the set of all equivalence classes of measurable sections of $\pi$, defined up to a.e equality.

Given two Lebesgue spaces $X,Y$ we use the term {\em Lebesgue morphism} to denote a Lebesgue map from $X$ to the underlying Borel space of $Y$ which is measure class preserving, that is the preimage of a null set in $Y$ is null in $X$.
There is a correspondence between Lebesgue morphisms $X\to Y$ and von Neumann algebra morphisms $L^\infty(Y)\to L^\infty(X)$.

Every lcsc group (that is, locally compact second countable topological group) gives rise to a Lebesgue space when endowed with its Haar measure class.
Accordingly, we will regard below many times lcsc groups as Lebesgue spaces without further mention.
Given a lcsc group $S$, a Lebesgue space $Y$ is called  an $S$-Lebesgue space
when it is endowed with a Lebesgue action of $S$. That is a Lebesgue morphism
$S\times Y\to Y$ which gives rise to an action of $S$ on $L^\infty(Y)$ by automorphisms of von Neumann algebras. 
A Borel space $V$ is called an $S$-Borel space when it endowed with a Borel action of $S$.
That is a Borel morphism $S\times V\to V$ which gives rise to an action of $S$ on $V$ by automorphisms of a Borel space.

\begin{defn}[Amenability, see {\cite[Definition~4.3.1]{zimmer-book}}]
Given a lcsc group $S$, an $S$-Lebesgue space $Y$ is called {\em amenable} if for every $S$-Borel space $V$
and an essentially surjective $S$-equivariant Borel map $\pi: V\to Y$ with a measurably defined compact convex structure on the fibers,
such that the $S$-action restricted to the fibers is by continuous affine maps,
one has $L^0(\pi)^S\neq \emptyset$. That is, every $S$-Borel bundle of convex compact sets over $S$ admits an invariant measurable section.
\end{defn}

%\begin{defn}[see {\cite[Definition~4.3.1]{zimmer-book}}]
%Given a lcsc group $S$, an $S$-space $Y$ is called {\em amenable} if for every measurable %bundle over $Y$ with compact convex fibers on which the $S$-action extends and the action on %the fibers is continuous and affine, there exists a measurable a.e defined $S$-invariant section.
%\end{defn}

\begin{defn}[Metric Ergodicity, see {\cite[\S2]{BF-lyap}}]
Given a lcsc group $S$, an $S$-Lebesgue space $Y$ is called {\em metrically ergodic} if for every separable metric space $V$ on which $S$ acts continuously by isometries, every $S$-equivariant Lebesgue map $Y\to V$ is a.e constant.
\end{defn}

\begin{theorem}[Kaimanovich--Zimmer] \label{boundary}
For every lcsc group $S$ there exists an $S$-Lebesgue space $Y$ which is both amenable and metrically ergodic.
\end{theorem}

\begin{remark}
In fact \cite{kaimanovich} shows that there exists a space $Y$ which is amenable and {\em doubly metrically ergodic}.
With a slight restriction, this was first proven in \cite{BM}.
A slightly stronger theorem is proven in \cite[Theorem~2.7]{BF-lyap}, but we will not discuss these extensions here.
\end{remark}

\begin{proof}[On the proof of Theorem~\ref{boundary}]
In \cite{kaimanovich} the weaker statement that every group possesses a strong boundary in the sense of \cite{BM} is proven,
but the same proof actually proves Theorem~\ref{boundary},
as explained in \cite[Remark~4.3]{GW}.
\end{proof}

\begin{lemma} \label{AMElattice}
Given a lcsc group $S$ and a lattice $\Gamma<S$, if $Y$ is an amenable and metrically ergodic $S$-Lebesgue space 
then it is also amenable and metrically ergodic as a $\Gamma$-Lebesgue space.
\end{lemma}

\begin{proof}
The fact that $Y$ is $\Gamma$-amenable follows from \cite[4.3.5]{zimmer-book}.
To show that $Y$ is $\Gamma$-metrically ergodic, we fix a measurable $\Gamma$-equivariant map $\phi:Y\to V$, 
where $(V,d)$ is a metric space on which $\Gamma$ acts isometrically, and argue to show that it is essentially constant.
Replacing $d(x,y)$ by $\min(d(x,y),1)$ we may assume $d$ to be bounded.
We consider the space $\LL(S,V)^\Gamma$ consisting of Lebesgue maps which are $\Gamma$-equivariant with respect 
to the left action of $\Gamma$ on $S$.
Given $\alpha,\beta\in \LL(S,V)^\Gamma$, note that $d(\alpha(x),\beta(x))$ is a $\Gamma$-invariant function on $S$;
that descends to a well define function on $\Gamma\backslash S$. 
Using this observation, we define on $\LL(S,V)^\Gamma$ a function $D$ by setting for $\alpha,\beta\in \LL(S,V)^\Gamma$,
\[ 
	D(\alpha,\beta)=\int_{\Gamma\backslash S} d(\alpha(x),\beta(x))\dd x. 
\]
Then $D$ is a metric on $\LL(S,V)^\Gamma$, and $S$ acts continuously and isometrically on this space 
via its right regular action on the domain. 
The map
\[ 
	\Phi:Y\to\LL(S,V)^\Gamma, \quad y \mapsto [s \mapsto \phi(sy)], 
\]
defined using Fubini's theorem (see \cite[Chapter VII, Lemma 1.3]{margulis-book}), is $S$-equivariant.
By the $S$-metric ergodicity of $Y$ we conclude that $\Phi$ is essentially constant,
and thus $\phi$ is essentially constant too.
\end{proof}

\begin{lemma} \label{AMEproduct}
Let $S_1,\dots,S_n$ be lcsc groups, and let $Y_i$ be an amenable and metrically ergodic 
$S_i$-Lebesgue space for each $i=1,\dots,n$. 
Then $Y_1\times \cdots \times Y_n$ is an amenable and metrically ergodic $S_1\times\cdots\times S_n$-Lebesgue space.
\end{lemma}

\begin{proof}
	By induction, it suffices to consider the case $n=2$.
The proof that $Y=Y_1\times Y_2$ is $S$-amenable for $S=S_1\times S_2$ is well known, so we merely sketch it.
Let $\pi:C\to Y$ be an
$S$-Borel bundle of convex compact sets over $Y$.
For every $y_1\in Y_1$ we let $C_{y_1}=\pi^{-1}(\{y_1\}\times Y_2)$ and
$\LL(\pi|_{C_{y_1}})$ be the corresponding space of sections.
We view these spaces as a convex compact bundle over $Y_1$ (using the obvious weak*-topology) 
and denote by $\LL(Y_1,\LL(\pi|_{C_{y_1}}))$ its space of sections.
Its obvious identification with $\LL(\pi)$ gives an identification
$\LL(\pi)^S\simeq \LL(Y_1,\LL(\pi|_{C_{y_1}})^{S_2})^{S_1}$.
The right hand side is non-empty by our amenability assumptions, hence so is the left hand side.

Let $V$ be an $S$-metric space and $\phi:Y\to V$  an $S$-equivariant Lebesgue map.
For a.e.\ $y_1\in Y_1$, $\phi_{\{y_1\}\times Y_2}$ is defined by Fubini's theorem, and it is essentially constant, 
as $Y_2$ is $S_2$-metrically ergodic.
Thus $\phi$ is reduced to a map $\phi':Y_1\to V$ which is again essentially constant, as $Y_1$ is $S_1$-metrically ergodic.
Thus $\phi$ is essentially constant.
\end{proof}

\begin{lemma} \label{me-e}
Let $S$ be a lcsc group and $X,Y$ be $S$-Lebesgue spaces.
Assume the action on $X$ is ergodic and probability measure preserving and the action on $Y$ is metrically ergodic.
Then the diagonal $S$-action on $X\times Y$ is ergodic.
\end{lemma}

\begin{proof}
For $f\in L^{\infty}(X\times Y)^S$, using Fubini's theorem, we define $F:Y\to L^\infty(X) \subset L^2(X)$
by $F(y)(x)=f(x,y)$.
$F$ is easily checked to be $S$-equivariant.
The image of $F$ must be $S$-invariant, by the metric ergodicity of $Y$, as the $S$-action on $L^2(X)$ is continuous.
By ergodicity of $X$ this image is a constant function, thus $f$ is constant.
\end{proof}

\begin{lemma} \label{T2B1ergodic}
	Let $S_1,\dots,S_n$ be lcsc groups, $\Gamma<S=S_1\times\cdots\times S_n$ be a lattice
	with dense projections, and let $Y_i$ be metrically ergodic 
	$S_i$-Lebesgue spaces for $i=1,\dots,n$.
	Then the $\Gamma$-action on $S_i\times \prod_{j\ne i}Y_j$ is ergodic.
\end{lemma}

\begin{proof}
We write $S=S_i\times S_i'$ where $S_i'=\prod_{j\ne i}S_j$.
Since $\pi_i(\Gamma)$ is dense in $S_i$, $\Gamma$ acts ergodically on $S_i=S/S_i'$.
Thus $S_i'$ acts ergodically on $X=S/\Gamma$.
Note that the latter action preserves a probability measure.
The diagonal $S_i'$-action on $Y_i'=\prod_{j\ne i}Y_j$ is metrically ergodic 
(Lemma~\ref{AMEproduct}), and using Lemma~\ref{me-e} we conclude that $X\times Y_i'$ is $S_i'$-ergodic.
It follows that $S\times Y_i'$ is $\Gamma\times S_i'$-ergodic, 
where $\Gamma$ acts by its left action on $S$ and $S_i'$ acts diagonally, 
via it right action on $S$ and its given action on $Y_i'$.
Note that the Lebesgue isomorphism $S\times Y_i' \to S\times Y_i'$, 
$(s,y) \mapsto (s,sy)$, 
% \[ 
% 	S\times Y_i' \to S\times Y_i',\qquad 
% \]
intertwines the above described action $\Gamma\times S_i'\acts S\times Y_i'$:
\[
	\gamma:(s,y)\mapsto (\gamma s,\gamma y),
	\qquad
	s':(s,y)\mapsto (ss'^{-1},y)
\] 
Taking the $S_i'$-orbit space, we conclude that $\Gamma$ acts ergodically on $S_i\times Y_i'$.
\end{proof}

We also make use of the following special situation.

\begin{lemma}\label{L:lifting}
	Let $S$ be a lcsc group, $X$ a Lebesgue $S$-space,
	$\Gamma$ a countable group, $\Gamma\to S$ a homomorphism,
	$Z$ a Lebesgue $\Gamma$-space.
	Let $V$ be a Borel $\Gamma$-space and $\phi:X\times Z\to V$ a measurable $\Gamma$-map.
	
	Then for a.e. $x\in X$ the map $\hat\phi_x:S\times Z\to V$ given by
	$\hat\phi_x(s,z)=\phi(sx, z)$ is a measurable $\Gamma$-map.
\end{lemma}
\begin{proof}
	Fix a Borel map $\phi$ in its class of equivalent measurable maps. The Borel set
	\[
		A=\{(x,z)\in X\times Z \mid \phi(\gamma x,\gamma z)=\gamma \phi(x,z),\ \gamma\in\Gamma\}
	\] 
	is co-null in $X\times Z$. So by Fubini's theorem (\cite[Theorem 8.19]{kechris}), 
	there is a co-null set $X_0\subset X$
	so that for each $x\in X_0$ for a.e. $z\in Z$ one has $(x,z)\in A$. 
	% Taking $X_1=\bigcap \gamma X_0$ we get a $\Gamma$-invariant set.
	Since the action map $S\times X\to X$ is non-singular, the preimage of $X_0$ 
	is co-null in $S\times X$, and so for a co-null
	set $X_1\subset X$ for $x\in X_1$ the set $S_x=\{ s\in S \mid sx\in X_0\}$ 
	is co-null in $S$.
	Therefore for every $x\in X_1$, $s\in S_x$, and a.e. $z\in Z$: 
	\[
		\hat\phi_x(\gamma s,\gamma z)=\phi((\gamma s) x,\gamma z)=\phi(\gamma (sx),\gamma z)
		=\gamma \phi(sx,z)=\hat\phi_x(s,z)
	\]
	for all $\gamma\in \Gamma$.
\end{proof}

\section{Algebraic representations of ergodic actions (AREA)} \label{sec:gate}

Throughout this section we fix the following data:
\begin{itemize}
\item a lcsc group $S$,
\item an ergodic $S$-Lebesgue space $Y$,
\item a field $k$ with a non-trivial absolute value which is separable and complete (as a metric space),
\item a $k$-algebraic group ${\bf G}$,
\item a continuous homomorphism $\rho:S\to {\bf G}(k)$,
where ${\bf G}(k)$ is regarded as a Polish group (see Proposition~\ref{polishing}).
\end{itemize}

\begin{defn}
Given all the data above, an \emph{algebraic representation} of $Y$
consists of the following data
\begin{itemize}
\item a $k$-${\bf G}$-algebraic variety ${\bf V}$,
\item an $S$-equivariant Lebesgue map $\phi:Y \to {\bf V}(k)$, where
${\bf V}(k)$ is regarded as a Polish space (see Proposition~\ref{polishing}).
\end{itemize}
%We regard two maps $Y \to {\bf V}(k)$ which agree a.e as the same representation.
Sometimes we abbreviate the notation by saying that ${\bf V}$ is an \emph{algebraic representation of} $Y$,
and denote $\phi$ by $\phi_{\bf V}$ for clarity.
A morphism from the algebraic representation ${\bf U}$ to the algebraic representation ${\bf V}$ consists of
\begin{itemize}
\item a $k$-algebraic map $\psi:{\bf U}\to {\bf V}$ which is ${\bf G}$-equivariant,
and such that $\phi_{\bf V}$ agrees almost everywhere with $\psi\circ \phi_{\bf U}$.
\end{itemize}
An algebraic representation ${\bf V}$ of $Y$ is said to be a \emph{coset algebraic representation}
if in addition
${\bf V}={\bf G}/{\bf H}$ for some $k$-algebraic subgroup ${\bf H}<{\bf G}$.
\end{defn}

\begin{prop} \label{coset contraction}
%Assume $Y$ is $T$ ergodic and
Let ${\bf V}$ be an algebraic representation of $Y$.
Then there exists a coset algebraic representation ${\bf G}/{\bf H}$
and a morphism of representations from ${\bf G}/{\bf H}$ to ${\bf V}$,
that is a $k$-${\bf G}$-algebraic map $i:{\bf G}/{\bf H}\to {\bf V}$
such that $\phi_{\bf V}=i\circ \phi_{{\bf G}/{\bf H}}$.
\end{prop}

This follows essentially from Proposition~\ref{polishing} together with the well known argument given in \cite[Lemma 5.2.11]{zimmer-book}.
For the reader's convenience we reproduce this argument below.

\begin{proof}
Denote $V={\bf V}(k)$ and $G={\bf G}(k)$.
By Proposition~\ref{polishing} we know that every $G$-orbit is locally closed in $G$.
Consider the orbit space
$V/G$ endowed with the quotient topology and Borel structure.
The map $Y\to V \to V/G$ is Borel.
We push the measure class given on $Y$ and obtain a measure class on $V/G$.
Since $V$ is Polish it has a countable basis thus so does $V/G$.
Let $\{U_n~|~n\in \bbN\}$ be sequence of subsets of $V/G$ consisting of the elements of a countable basis and their complements.
Set
\[ U=\bigcap \{U_n~|~n\in \bbN,~U_n \mbox{ has a full measure in } V/G \}. \]
Then $U$ has a full measure and in particular it is non-empty.
We claim that $U$ is a singleton.
Indeed,
since every $G$-orbit is locally closed in $V$, the quotient topology of $V/G$ is $T_0$,
thus if $U$ would contain two distinct points we could find a basis set $U_n$ which separates them,
but by the ergodicity of $G$ on $Y$ either it or its complement would be of full measure, which contradicts the definition of $U$.

Fixing $v\in V$ which is in the preimage of $U$ we conclude that $\phi(X)$ is essentially contained in $Gv$.
Let ${\bf H}<{\bf G}$ be the stabilizer of $v$ and $H={\bf H}(k)$.
By Proposition~\ref{polishing} we get
a $k$-algebraic map $i:{\bf G}/{\bf H}\to {\bf V}$
whose restriction to $G/H$ gives a homeomorphism with the
orbit $Gv$.
We let $\phi_{{\bf G}/{\bf H}}=(i|_{G/H})^{-1} \circ \phi$.
We are done by extending the codomain of $\phi_{{\bf G}/{\bf H}}$ to ${\bf G}/{\bf H}(k)$ via the embedding $G/H \hookrightarrow {\bf G}/{\bf H}(k)$.
\end{proof}

%It is clear that the collection of algebraic representations of $X$ and their morphisms form a category.
In \cite[Definition~9.2.2]{zimmer-book}, following Mackey,
Zimmer defined the notion ``algebraic hull of a cocycle".
We will not discuss this notion here,
but we do point out its close relation with the following theorem (to be precise, it coincides with the group ${\bf H}_0$
appearing in the proof below).

\begin{theorem}[cf.\ {\cite[Proposition~9.2.1]{zimmer-book}}] \label{thm:gate}
%Assume $Y$ is $T$ ergodic.
The category of algebraic representations of $Y$ has an initial object.
Moreover, there exists an initial object which is a coset algebraic representation.
\end{theorem}

\begin{proof}
We consider the collection
\[ \{{\bf H}<{\bf G}~|~{\bf H}\mbox{ is defined over } k \mbox{ and there exists a coset representation to } {\bf G}/{\bf H} \}. \]
This is a non-empty collection as it contains ${\bf G}$.
By the Neotherian property, this collection contains a minimal element.
We choose such a minimal element ${\bf H}_0$
and fix corresponding
$\phi_0:Y \to ({\bf G}/{\bf H}_0)(k)$.
We argue to show that this coset representation is the required initial object.

Fix any algebraic representation of $Y$, ${\bf V}$.
It is clear that, if it exists, a morphism of algebraic representations from ${\bf G}/{\bf H}_0$ to ${\bf V}$ is unique, as two different ${\bf G}$-maps
${\bf G}/{\bf H}_0\to {\bf V}$ agree nowhere.
We are left to show existence.
To this end we consider
the product representation ${\bf V}\times {\bf G}/{\bf H}_0$ given by $\phi=\phi_{\bf V}\times \phi_0$.
Applying Proposition~\ref{coset contraction} to this product representation we obtain the commutative diagram

%\medskip

\begin{equation} \label{diag-AG}
\xymatrix{ Y \ar@{.>}[r] \ar[d]^{\phi_{\bf V}} \ar[rd]^{\phi} \ar@/^3pc/[rrd]^{\phi_0} & {\bf G}/{\bf H} \ar@{.>}[d]_{i} &  \\
		   {\bf V} & {\bf V}\times {\bf G}/{\bf H}_0 \ar[r]^{~~p_2} \ar[l]_{p_1~~~~} & {\bf G}/{\bf H}_0  }
\end{equation}

By the minimality of ${\bf H}_0$, the ${\bf G}$-morphism $p_2\circ i:{\bf G}/{\bf H} \to {\bf G}/{\bf H}_0$ must be a $k$-isomorphism.
We thus obtain the $k$-${\bf G}$-morphism
\[ p_1\circ i \circ (p_2\circ i)^{-1}:{\bf G}/{\bf H}_0 \to {\bf V}. \]
\end{proof}

\begin{defn}[Algebraic Gate] \label{def:gate}
With a slight abuse of terminology (as it is not canonical), we call (a choice of) a coset representation 
$Y\to {\bf G}/{\bf H}(k)$ which is initial as an algebraic representation of $Y$, 
{\em the algebraic gate} of $Y$. 
In case ${\bf H}\lneq {\bf G}$ we say that the algebraic gate of $Y$ is non-trivial.
\end{defn}

The notion of the algebraic gate and the applications that we will derive below are only interesting if the gate is non-trivial.
The following theorem gives a criterion for non-triviality.

\begin{theorem} \label{AMEnontriv}
Assume the $S$-Lebesgue space $Y$ is both amenable and metrically ergodic.
Assume the $k$-algebraic group ${\bf G}$ is connected, $k$-simple and adjoint
and assume that $\rho(S)$ is unbounded in ${\bf G}(k)$.
Then the gate of $Y$ is non-trivial.
\end{theorem}

\begin{proof}
Using the amenability of $Y$, it follows from \cite[Corollary~1.17]{BDL} that if the gate of $Y$ is non-trivial then there exists a separable metric space $V$ on which ${\bf G}(k)$ acts by isometries and with bounded stabilizers and an $S$-equivariant map $Y\to V$.
But the latter possibility is ruled out by the assumption that $Y$ is $S$-metrically ergodic 
and the assumption that $\rho(S)<{\bf G}(k)$ is unbounded.
\end{proof}

Before proceeding to our next theorem,
let us state without a proof the following proposition
which
provides an identification of $\Aut_{\bf G}({\bf G}/{\bf H})$ that we will keep using throughout the paper.
The proposition is well known and easy to prove.

\begin{prop} \label{aut-identification}
Fix a $k$-subgroup ${\bf H}<{\bf G}$ and denote ${\bf N}=N_{\bf G}({\bf H})$.
This is again a $k$-subgroup.
Any element $n\in {\bf N}$ gives a ${\bf G}$-automorphism of ${\bf G}/{\bf H}$ by
$g{\bf H}\mapsto gn^{-1}{\bf H}$.
The homomorphism $ {\bf N} \to \Aut_{\bf G}({\bf G}/{\bf H})$ thus obtained is onto and its kernel is ${\bf H}$.
Under the obtained identification ${\bf N}/{\bf H} \simeq \Aut_{\bf G}({\bf G}/{\bf H})$,
the $k$-points of the $k$-group ${\bf N}/{\bf H}$ are identified with the $k$-${\bf G}$-automorphisms of ${\bf G}/{\bf H}$.
%that is ${\bf N}/{\bf H}(k)\simeq \Aut_{\bf G}({\bf G}/{\bf H})$.
\end{prop}

The following theorem provides a most useful tool in the study of rigidity.

\begin{theorem} \label{yoneda}
Assume $\phi:Y\to {\bf G}/{\bf H}(k)$ is the algebraic gate of $Y$.
Let $S'$ be a lcsc group which acts on $Y$ commuting with the $S$-action.
Then there exists a unique homomorphism 
\[	
	\rho':S'\to N_{\bf G}({\bf H})/{\bf H}(k)
\]
which turns $\phi$ into an $S\times S'$-equivariant map, where $S\times S'$ acts on the codomain via $\rho\times \rho'$.
This homomorphism $\rho'$ is continuous.
\end{theorem}

\begin{proof}
For a given $s'\in S'$ we consider the diagram
\begin{equation} \label{a-diag-AG}
\xymatrix{ Y \ar[r]^{\phi} \ar[d]^{s'} & {\bf G}/{\bf H} \ar@{.>}[d]^{\rho'(s')}  \\
		 Y \ar[r]^{\phi}   & {\bf G}/{\bf H}  }
\end{equation}
where we denote by $\rho'(s')$ the dashed arrow, which is the unique $k$-algebraic ${\bf G}$-equivariant morphism ${\bf G}/{\bf H} \to {\bf G}/{\bf H}$ given by the fact that $\phi$ is a gate map and $\phi\circ s'$ is an algebraic representation of $Y$.
By the uniqueness of the dashed arrow, the correspondence $s'\mapsto \rho'(s')$ is easily checked to form a homomorphism from $S'$ to the group of $k$-${\bf G}$-automorphisms of ${\bf G}/{\bf H}$ which we identify with $ N_{\bf G}({\bf H})/{\bf H}(k)$
using Proposition~\ref{aut-identification}.
We are left to check the continuity of $\rho'$.

To simplify the notations we let $V={\bf G}/{\bf H}(k)$, $M=N_{\bf G}({\bf H})/{\bf H}(k)$
and $U=\LL(Y,V)$.
We endow $U$ with the action of $M$ by post-composition, and the action of $S'$  by precomposition.
By the fact that $M$ acts freely on $V$, we get that $M$ acts freely on $U$ as well.
Using Proposition~\ref{polishing},
\cite[Proposition~3.3.1]{zimmer-book} gives that the $M$-action on $U$ has locally closed orbits\footnote{Actually, in
\cite[Proposition~3.3.1]{zimmer-book} it is assumed that $k$ is a local field of zero characteristic, as it relies on \cite[Theorem~3.1.3]{zimmer-book}, but upon replacing \cite[Theorem~3.1.3]{zimmer-book} with  Proposition~\ref{polishing} the proof applies verbatim here as well.}.
It follows that the $M$-orbit map $M\to U$, $m\mapsto m\circ \phi$ is a homeomorphism onto its image, $M\phi$.
We let $\alpha:M\phi\to M$ be its inverse.

By the fact that the map $S'\times Y \to V$, $(s',y)\mapsto \phi(s'y)$ is a.e defined and measurable, we get that the associated map $\beta:S'\to U$, $s'\mapsto \phi\circ s'$ is a.e defined and measurable, see  \cite[Chapter VII, Lemma 1.3]{margulis-book}.
Since by the definition of $\rho'$, $\phi\circ s'=\rho'(s')\circ \phi$, we conclude that $\rho'$ agrees a.e with $\alpha\circ \beta$ which is a.e defined and measurable.
It follows that $\rho'$ is measurable.
By \cite[Lemma 2.1]{Rosendal} we conclude that $\rho'$ is a continuous homomorphism.
\end{proof}

%%%%%%%%%%%%%%%%%

\section{Extension of homomorphisms defined on dense subgroups} \label{sec:extension}

Throughout this section we fix
\begin{itemize}
\item lcsc groups $S$ and $S'$ and a continuous homomorphism $\theta:S\to S'$ such that $\theta(S)$ is dense in $S'$,
\item a field $k$ with a non-trivial absolute value which is separable and complete (as a metric space),
\item a $k$-algebraic group ${\bf G}$,
\item a continuous homomorphism $\rho:S\to {\bf G}(k)$,
where ${\bf G}(k)$ is regarded as a Polish group (see Proposition~\ref{polishing}),
such that $\rho(S)$ is Zariski-dense in ${\bf G}$.
\end{itemize}
We will explain how under some assumptions the homomorphism $\rho$ extends to $S'$ via $\theta$.
The main result of this section is the following theorem.

\begin{theorem} \label{thm:extension}
	Assume ${\bf G}$ is a connected, adjoint, $k$-simple group
	and let ${\bf V}$ be a $k$-${\bf G}$-variety which has no ${\bf G}$-fixed point.
	Let $Z$ be an $S'$-Lebesgue space.
	Consider $Z$ as an $S$-Lebesgue space via $\theta$ and assume that there exists an $S$-equivariant Lebesgue map
	$\sigma:Z\to {\bf V}(k)$.
	Then there exists a continuous homomorphism $\bar{\rho}:S'\to {\bf G}(k)$ such that $\rho=\bar{\rho}\circ\theta$.
\end{theorem}

The proof will rely on the following lemma.

\begin{lemma} \label{lem:extension}
	Assume that in addition to $\rho:S\to {\bf G}(k)$ we are also given a continuous homomorphism $\rho':S'\to {\bf G}(k)$ 
	and a Lebesgue map $\phi:S'\to {\bf G}(k)$ which is $S\times S'$-equivariant with respect to the 
	left $S$- and right $S'$-actions on $S'$ and their corresponding actions on ${\bf G}(k)$ via $\rho\times \rho'$.
	Then $\rho=\inn(g)\circ\rho'\circ \theta$ for some $g\in {\bf G}(k)$,
	where $\inn(g)$ denotes the corresponding inner automorphism of ${\bf G}(k)$.
\end{lemma}

\begin{proof}[Proof of the lemma]
The map $\psi:S'\to {\bf G}(k)$ given by $s'\mapsto \phi(s')\rho'(s')^{-1}$ is $S'$-invariant, hence constant a.e. 
We let $g$ be its essential value. 
Given $s\in S$ we pick $s'\in S'$ such that both $s'$ and $\theta(s)s'$ are both $\phi$ and $\psi$-generic. 
We get
\[ 
	\begin{split}
		\rho(s)\phi(s') &=\phi(\theta(s)s')=\psi(\theta(s)s')\rho'(\theta(s)s')\\
		&=g\rho'(\theta(s))\rho'(s') =g\rho'(\theta(s))\psi(s')^{-1}\phi(s')\\
		&=g\rho'(\theta(s))g^{-1}\phi(s'), 
	\end{split}
\]
and conclude that $\rho(s)=\inn(g)\circ\rho'\circ \theta(s)$.
\end{proof}

\begin{proof}[Proof of {Theorem~\ref{thm:extension}}:]
First, we replace the unknown $S'$-Lebesgue space $Z$ with the group $S'$ endowed with its Haar measure. 
We do it by restricting $\sigma$ to a generic orbit.
More pedantically, we argue as follows.
We consider the map
\[ 
	S'\times Z \to {\bf V}(k), \quad (s',z) \mapsto \sigma(s'z) 
\]
and use Fubini's theorem in order to fix a generic point $z\in Z$ such that the map $s' \to \sigma(s'z)$ is a.e defined on $S'$. 
We observe that the latter map $S'\to {\bf V}(k)$ is $S$-equivariant.
Replacing $\sigma$ with this map, we assume as we may that $Z=S'$ and $\sigma:S'\to {\bf V}$ is an $S$-equivariant Lebesgue map.

Note that $S'$ is $S$-ergodic, as $\theta(S)$ is dense in $S'$.
Therefore we may use \S\ref{sec:gate} in the setting $Y=S'$ and view $\sigma$ as an algebraic representation of the $S$-space $S'$.
By Theorem~\ref{thm:gate} we obtain the algebraic gate of $S'$, $\phi:S'\to {\bf G}/{\bf H}(k)$ for some $k$-algebraic group 
${\bf H}\leq {\bf G}$.
Since there exists a ${\bf G}$-equivariant $k$-morphism ${\bf G}/{\bf H}\to{\bf V}$ and ${\bf V}$ has no ${\bf G}$-fixed point, 
we conclude that ${\bf H}\lneq {\bf G}$.

We now consider the right action of the group $S'$ on the space $S'$ and conclude by Theorem~\ref{yoneda} the existence of a 
continuous homomorphism $\rho':S'\to {\bf N}/{\bf H}(k)$, where ${\bf N}=N_{\bf G}({\bf H})$, 
such that $\phi$ is $S\times S'$-equivariant via $\rho\times\rho'$.

We claim that ${\bf H}=\{e\}$. Assume not.
Then, by the fact that ${\bf G}$ is simple we conclude that ${\bf N}\lneq {\bf G}$.
Composing $\phi$ with the map ${\bf G}/{\bf H}(k)\to {\bf G}/{\bf N}(k)$ we get a map $S'\to {\bf G}/{\bf N}(k)$ 
which is left $S$-equivariant and right $S'$-invariant. By its right $S'$-invariance, it is essentially constant. 
By the left $S$-equivariance, its essential image is $\rho(S)$-invariant, hence also ${\bf G}(k)$-invariant, 
as $\rho(S)$ is Zariski-dense in ${\bf G}$. This is absurd, thus the claim is proven.

It follows from the claim that ${\bf N}={\bf G}$ and thus $\rho'$ is a continuous homomorphism from $S'$ to ${\bf G}(k)$ 
and $\phi$ is a map from $S'$ to ${\bf G}(k)$ which is $S\times S'$-equivariant via 
$\rho\times \rho':S\times S' \to {\bf G}(k)\times {\bf G}(k)$. 
We are in a situation to apply Lemma~\ref{lem:extension} and obtain $g\in {\bf G}(k)$ such that $\rho=\inn(g)\circ\rho'\circ\theta$. 
We are done by setting $\bar{\rho}=\inn(g)\circ\rho'$.
\end{proof}

\section{Proof of Theorem~\ref{thm:lattice}} \label{section:mainthm}

Below we give a proof of the main theorem.
A priori, the valued field $k$ is assumed merely to be a complete field, but by the fact that $\rho(\Gamma)$ is unbounded 
we can assume that the absolute value on $k$ is non-trivial.
Further, by the countability of $\Gamma$, we may replace $k$ with a complete and separable (in the topological sense) 
subfield $k'$ such that $\rho(\Gamma)\subset {\bf G}(k')$. 
Therefore hereafter we assume that the given absolute value on the field $k$ is non-trivial 
and that $k$ is complete and separable as a metric space.
We start with the proof of the existence of a continuous homomorphism $\bar{\rho}:T\to {\bf G}(k)$;
the uniqueness will follow from the general Lemma~\ref{lem:uniquness} given below.

\medskip

Using Theorem~\ref{boundary}, for each $T_i$ one can choose a Lebesgue $T_i$-space $B_i$ which is amenable and metrically ergodic.
By Lemma~\ref{AMEproduct} the product $B=B_1\times \cdots\times B_n$ is amenable and metrically ergodic as a $T$-Lebesgue space, 
and Lemma~\ref{AMElattice} implies that $B$ is also amenable and metrically ergodic as a $\Gamma$-Lebesgue space.

We write $[n]=\{1,\dots,n\}$ and for a subset $I\subset [n]$ denote by $T_I=\prod_{i\in I} T_i$ the 
factor of the group $T$, and by $B_I=\prod_{i\in I} B_i$ the measurable factor of the Lebesgue space $B$.
By convention, $T_\emptyset=\{e\}$ and $B_\emptyset=\{*\}$. 

Consider pairs $(i,J)$ where $i\in [n]$, $J\subset [n]\setminus\{i\}$, and view $T_i\times B_J$
as a Lebesgue space with commuting actions of $\Gamma$ and $T_i$:
\[
	\Gamma \ni \gamma:(s,b)\mapsto (\gamma s, \gamma b),\qquad T\ni t:(s,b)\mapsto (s t^{-1},b).
\]
Lemma~\ref{T2B1ergodic} shows that the $\Gamma$-action on $T_i\times B_{[n]\setminus \{i\}}$, 
and hence the $\Gamma$-action on its quotient $T_i\times B_J$, is ergodic.
Therefore we can apply the considerations of \S\ref{sec:gate} to the group 
$S=\Gamma$ acting on the space $Y=T_i\times B_J$.
Theorem~\ref{thm:gate} provides the corresponding algebraic gate: 
a minimal $k$-algebraic subgroup $\mathbf{H}_{i,J}<\mathbf{G}$
for which there exists a $\Gamma$-map $\phi_{i,J}:T_i\times B_J\to \mathbf{G}/\mathbf{H}_{i,J}(k)$.

Theorem~\ref{AMEnontriv} ensures that the algebraic gate of $B=B_{[n]}$ is non-trivial:
there is a proper $k$-subgroup $\mathbf{H}\ne \mathbf{G}$ and a measurable
$\Gamma$-map $B\to \mathbf{G}/\mathbf{H}(k)$.
For any $i\in [n]$ we can apply Lemma~\ref{L:lifting} in the setting $S=T_i$, $X=B_i$ and 
$Z=B_{[n]\setminus \{i\}}$
to obtain a measurable $\Gamma$-map 
$T_i\times B_{[n]\setminus \{i\}}\to \mathbf{G}/\mathbf{H}(k)$.
It follows that $\mathbf{H}_{i,[n]\setminus \{i\}}$ can be conjugated into $\mathbf{H}$, and is therefore a proper subgroup of $\mathbf{G}$.

Let $J\subset [n]$ be a subset of minimal size $|J|$ for which there exists $i\in [n]\setminus J$
so that the algebraic gate of $T_i\times B_J$ is non-trivial:
\[
	\phi_{i,J}:T_i\times B_J\to \mathbf{G}/\mathbf{H}_{i,J}(k),\qquad \mathbf{H}_{i,J}\neq \mathbf{G}.
\]
Since $T_i$ acts on $T_i\times B_J$ by right translation on the first coordinate, commuting with the ergodic $\Gamma$-action,
Lemma~\ref{yoneda} with $S=\Gamma$, $Y=T_i\times B_J$ and $S'=T_i$ yields a continuous homomorphism 
\[
	\sigma:T_i \overto{}\mathbf{N}_\mathbf{G}(\mathbf{H}_{i,J})/\mathbf{H}_{i,J}(k)
\]
so that $\phi_{i,J}(\gamma x t,\gamma b)=\rho(\gamma) \phi_{i,J}(x,b) \sigma(t)$
for a.e. $(x,b)\in T_i\times B_J$.
Let $\mathbf{L}$ denote the smallest $k$-algebraic subgroup in $\mathbf{N}_\mathbf{G}(\mathbf{H}_{i,J})$
for which $\mathbf{L}/\mathbf{H}_{i,J}(k)$ contains the image $\sigma(T_i)$.

\begin{claim}
	$\mathbf{L}=\mathbf{G}$ and $\mathbf{H}_{i,J}=\{e\}$.
\end{claim}
Once the claim is proven, we get that there is a continuous homomorphism $\sigma:T_i\to \mathbf{G}(k)$
with Zariski dense image, and a measurable map $\phi:T_i\times B_J\to \mathbf{G}(k)$ so that
\[
	\phi(\gamma x t,\gamma b)=\rho(\gamma) \phi(x,b) \sigma(t)
\]
for a.e. $(x,b)\in T_i\times B_J$.
\begin{proof}
	Assume $\mathbf{L}\ne \mathbf{G}$. 
	The $T_i$-equivariance property of $\phi_{i,J}$ implies that this map descends to a measurable $\Gamma$-map
	\[
		B_J\ \overto{}\ \mathbf{G}/\mathbf{L}(k).
	\]
	Note that this is possible only if $J\ne\emptyset$. Pick $j\in J$ and use Lemma~\ref{L:lifting} to get a measurable
	$\Gamma$-map $\psi:T_j\times B_{J\setminus\{j\}}\to \mathbf{G}/\mathbf{L}(k)$.
	Under the assumption that $\mathbf{L}$ is a proper subgroup in $\mathbf{G}$, 
	the gate of $T_j\times B_{J\setminus\{j\}}$ is non-trivial, contradicting the minimality property of $J\subset [n]$.
	
	Therefore $\mathbf{L}=\mathbf{G}$. This implies that $\mathbf{N}_\mathbf{G}(\mathbf{H}_{i,J})=\mathbf{G}$,
	meaning that $\mathbf{H}_{i,J}$ is normal in the simple group $\mathbf{G}$. 
	Since $\mathbf{H}_{i,J}$ is known to be a proper subgroup, we have $\mathbf{H}=\{e\}$.
	Thus the previously constructed continuous homomorphism $\sigma$ and the measurable map $\phi=\phi_{i,J}$ 
	range into $\mathbf{G}(k)$.
	Finally, $\sigma(T_i)$ is Zariski dense in $\mathbf{G}$, because its Zariski closure is $\mathbf{L}$. 
\end{proof}

\begin{claim}
	The $\Gamma$-map $\phi:T_i\times B_J\to \mathbf{G}(k)$ descends to a $\Gamma$-map $T_i\to \mathbf{G}(k)$.
\end{claim}
\begin{proof}
	Let us show that for every $j\in J$ for a.e. $t\in T_j$ one has $\phi(x,b)=\phi(x,tb)$ a.e. on $T_i\times B_J$.
	Since $T_J$ acts ergodically on $B_J$ the claim would follow.
	The case of $J=\emptyset$ being trivial, we assume $J$ is non-empty, fix $j\in J$, and set $K=J\setminus \{j\}$.
	Choose a.e. $b\in B_J$ to create a $\Gamma$-equivariant lift by Lemma~\ref{L:lifting}
	\[
		\hat\phi:T_i\times T_j\times B_K\ \overto{}\ \mathbf{G}(k).
	\]
	Consider the setting of \S\ref{sec:gate} 
	for the group $S=\Gamma\times T_i$, the homomorphism
	\[
		\rho\times\sigma:\Gamma\times T_i\overto{} \mathbf{G}(k)\times \mathbf{G}(k),
	\]
	and the Lebesgue $S$-space $Y=T_i\times T_j\times B_K$ with the action:
	\[
		(\gamma,t):(x,y,b)\mapsto (\gamma x t^{-1},\gamma y, \gamma b).
	\]
	Note that the $S$-action on $Y$ is ergodic, because the $\Gamma$-action on $T_j\times B_K$
	is ergodic by Lemma~\ref{T2B1ergodic}.
	Denote $\Delta_\mathbf{G}<\mathbf{G}\times \mathbf{G}$ the diagonal subgroup.
	The measurable map 
	\[
		\psi:Y\to \mathbf{G}\times \mathbf{G}/\Delta_\mathbf{G}(k),
		\qquad \psi(x,y,b)=(\hat\phi(x,y,b),\sigma(x))\Delta_\mathbf{G}(k)
	\]
	is $\rho\times \sigma$-equivariant. Thus the gate of $Y$ is given by  
	\begin{itemize}
		\item a $k$-variety 
		$\mathbf{V}=\mathbf{G}\times \mathbf{G}/\mathbf{D}$ 
		where $\mathbf{D}$ is a $k$-algebraic subgroup $\mathbf{D}<\Delta_\mathbf{G}$,
		\item 
		and a measurable $\Gamma\times T_i$-map $\hat{\psi}:Y\to \mathbf{V}(k)$ for which
		$\psi=\pi\circ \hat{\psi}$ where $\pi:\mathbf{V}(k)\to \mathbf{G}\times \mathbf{G}/\Delta_\mathbf{G}(k)$ is the projection.
	\end{itemize}
	In fact, we must have $\mathbf{D}=\Delta_\mathbf{G}$.
	Indeed, otherwise under the projection to the first factor
	$\pi_1:\mathbf{G}\times \mathbf{G}\to \mathbf{G}$, the image $\mathbf{M}=\pi_1(\mathbf{D})$ would be
	a proper subgroup of $\mathbf{G}$, and the map 
	\[
		T_i\times T_j\times B_K\ \overto{\hat{\psi}}\ \mathbf{V}(k)\ \overto{\pi_1}\ \mathbf{G}/\mathbf{M}(k)
	\]
	would factor through a measurable $\Gamma$-map $T_j\times B_K\to \mathbf{G}/\mathbf{M}(k)$,
	contrary to our assumption of minimality for $J\subset [n]$.
	Thus $\mathbf{D}=\Delta_\mathbf{G}$ and $\hat\psi=\psi$.

	Consider the action of $S'=T_j$ on $Y=T_i\times T_j\times B_K$ by
	$s:(x,y,b)\mapsto (x,ys^{-1},b)$. It commutes with the $S=\Gamma\times T_i$-action.
	We can apply Lemma~\ref{yoneda} to obtain a continuous homomorphism 
	$\tau:T_j\to N_{\mathbf{G}\times \mathbf{G}}(\Delta_\mathbf{G})/\Delta_\mathbf{G}(k)$
	for which $\psi=\hat\psi$ is $T_j$-equivariant. 
	But as  $\mathbf{G}$ is center-free, 
	$\Delta_\mathbf{G}$ is its own normalizer in $\mathbf{G}\times \mathbf{G}$.
	So $\tau$ is trivial, and therefore $\psi$ is $T_j$-invariant.
	Since $\psi$ was a lift of $\phi:T_i\times B_J\to \mathbf{G}(k)$ for a.e. $b\in B_J$,
	we proved that $\phi$ is essentially $T_j$-invariant, and this is true for all $j\in J$.
	This completes the proof of the claim. 
\end{proof}
At this point we have a measurable map $\phi:T_i\to \mathbf{G}(k)$ and a continuous homomorphism 
$\sigma:T_i\to \mathbf{G}(k)$ so that $\phi(\gamma t)=\rho(\gamma)\phi(t)$ and $\phi(tt')=\phi(t)\sigma(t')$.
Applying Lemma~\ref{lem:extension} with $\theta:\Gamma\overto{} T\overto{\pi_i} T_i$ 
we deduce that for some $g\in \mathbf{G}(k)$ the homomorphism
$\bar\rho_i=\inn(g)\circ \sigma:T_i\overto{} \mathbf{G}$
extends $\rho:\Gamma\to \mathbf{G}(k)$ in the sense that $\bar{\rho}_i\circ \theta=\rho$.
This completes the proof of the existence of the continuous extension homomorphism $\bar\rho:T\to \mathbf{G}(k)$.

The uniqueness part follows from the following general lemma, by which we conclude our proof.

\begin{lemma} \label{lem:uniquness}
Let $S$ be a lcsc group and $\Gamma<S$ a lattice.
Let $k$ be a complete valued field and ${\bf G}$ be a connected $k$-simple adjoint algebraic group.
Let $\rho_1,\rho_2:S \to {\bf G}(k)$ be continuous homomorphisms, and assume that $\rho_1$ has an unbounded and Zariski dense image.
If $\rho_1|_\Gamma=\rho_2|_\Gamma$ then $\rho_1=\rho_2$.
\end{lemma}

\begin{proof}
As before, without loss of generality we assume that the absolute value on $k$ is non-trivial and as a metric space $k$ is complete and separable.
We will prove the lemma by contradiction, assuming there exists $s_0\in S$ with $\rho_1(s_0)\neq \rho_2(s_0)$. 

We will first show that $\rho_1(\Gamma)=\rho_2(\Gamma)$ is Zariski dense in $\mathbf{G}$.
We denote by $L_1$ the closure of $\rho_1(S)$ in $\mathbf{G}(k)$.
By \cite{BDL}*{Proposition 1.9}
there exists a $k$-subgroup $\mathbf{H}_1<\mathbf{G}$ which is normalized by $L_1$ such that $L_1$ has a precompact image in 
$N_\mathbf{G}(\mathbf{H}_1)/\mathbf{H}_1(k)$ and such that for every $k$-$\mathbf{G}$-variety $\mathbf{V}$,
any $L_1$-invariant finite measure is supported on the subveriety of $\mathbf{H}_1$-fixed points, $\mathbf{V}^{\mathbf{H}_1}\cap \mathbf{V}(k)$.
As $\rho_1$ has a Zariski dense image, $L_1$ is Zariski dense in $\mathbf{G}$ and we get that $\mathbf{H}_1$ is normal in $\mathbf{G}$.
As $\mathbf{G}$ is $k$-simple, we conclude that either $\mathbf{H}_1=\{e\}$ or $\mathbf{H}_1=\mathbf{G}$.
The first option is ruled out, as $\rho_1(S)$ is unbounded while $L_1$ has a precompact image in 
$N_\mathbf{G}(\mathbf{H}_1)/\mathbf{H}_1(k)$.
Thus $\mathbf{H}_1=\mathbf{G}$.
We denote by $\mathbf{G}_1$ the Zariski closure of $\rho_1(\Gamma)$ and let $\mathbf{V}_1=\mathbf{G}/\mathbf{G}_1$.
The map $\rho_1$ gives rise to an obvious map $S/\Gamma \to \mathbf{V}_1(k)$.
By pushing the $S$-invariant measure of $S/\Gamma$ we obtain an $L_1$-invariant finite measure on $\mathbf{V}_1(k)$.
It follows that this measure is supported on the subveriety of $\mathbf{G}$-fixed points.
In particular, there exists a $\mathbf{G}$-fixed point in $\mathbf{V}_1=\mathbf{G}/\mathbf{G}_1$
and we conclude that $\mathbf{G}_1=\mathbf{G}$.
Thus, indeed, $\rho_1(\Gamma)$ is Zariski dense in $\mathbf{G}$.

Next we consider the homomorphism $R=\rho_1\times \rho_2:S\to \mathbf{G}(k)\times \mathbf{G}(k)$,
we let $L$ be the closure of $R(S)$ and let ${\bf L}<{\bf G}\times {\bf G}$ be the Zariski closure of $L$.
We note that, by the discussion above, the Zariski closure of $R(\Gamma)$ coincides with the diagonal group $\Delta_{\bf G}<{\bf G}\times {\bf G}$.
In particular, $\Delta_{\bf G}$ is contained in ${\bf L}$.
We argue to show that ${\bf L}={\bf G}\times {\bf G}$.
The $k$-subgroup ${\bf L}\cap ({\bf G}\times \{e\})$ is non-trivial, as the element $\rho_1(s_0)\rho_2(s_0)^{-1}$ is contained in its group of $k$-points, since both $(\rho_1(s_0),\rho_2(s_0))$ and $(\rho_2(s_0),\rho_2(s_0))$ are elements of ${\bf L}(k)$.
As this subgroup is normalized by ${\bf L}$, hence by $\Delta_{\bf G}$, we get that this is a non-trivial normal $k$-subgroup of ${\bf G}\times \{e\}$.
As ${\bf G}\times \{e\}$ is $k$-simple, we conclude that ${\bf L}\cap ({\bf G}\times \{e\})={\bf G}\times \{e\}$,
that is ${\bf G}\times \{e\}<{\bf L}$.
As ${\bf G}\times {\bf G}$ is generated by ${\bf G}\times \{e\}$ and $\Delta_{\bf G}$, we conclude that indeed ${\bf L}={\bf G}\times {\bf G}$.

We use \cite{BDL}*{Proposition 1.9} again, this time for the $k$-group ${\bf G}\times {\bf G}$ and the Zariski dense closed subgroup $L<({\bf G}\times {\bf G})(k)$,
and conclude having a normal $k$-subgroup ${\bf H}<{\bf G}\times {\bf G}$ such that $L$ has a precompact image in 
$({\bf G}\times {\bf G})/\mathbf{H}(k)$ and such that for every $k$-$\mathbf{G}$-variety $\mathbf{V}$,
any $L$-invariant finite measure is supported on the subveriety of $\mathbf{H}$-fixed points, $\mathbf{V}^{\mathbf{H}}\cap \mathbf{V}(k)$.
We	consider ${\bf V}=({\bf G}\times {\bf G})/\Delta_{\bf G}$.
The map $R$ gives rise to an obvious map $S/\Gamma \to \mathbf{V}(k)$.
By pushing the $S$-invariant measure of $S/\Gamma$ we obtain an $L$-invariant finite measure on $\mathbf{V}(k)$.
It follows that this measure is supported on the subveriety of $\mathbf{H}$-fixed points.
Since the non-trivial normal $k$-subgroups of ${\bf G}\times {\bf G}$, namely ${\bf G}\times \{e\}$, $\{e\}\times {\bf G}$
and ${\bf G}\times {\bf G}$ itself, have no fixed points in ${\bf V}$, we conclude that ${\bf H}=\{e\}$.
It follows that $L$ is compact in ${\bf G}\times {\bf G}$, which contradicts the assumption that $\rho_1$ has an unbounded image.
\end{proof}

%%%%%%%%%%%%%%%%%%%%%%%%%%%%

\section{Linearity criteria} \label{section:nonlinearity}

This section is devoted to the proof of the following theorem.

\begin{theorem} \label{lem:nonlinearity}
For any countable group $\Lambda$ and a commensurated subgroup $\Gamma<\Lambda$ the following properties are equivalent.
\begin{enumerate}

\item For every field $K$, integer $d$ and a group homomorphism $\phi:\Lambda \to \GL_d(K)$, $\phi(\Gamma)$ is solvable-by-locally finite.

\item For every finite index subgroup $\Lambda'<\Lambda$, complete field with an absolute value $k$, connected adjoint $k$-simple algebraic group ${\bf G}$ and Zariski dense group homomorphism $\rho:\Lambda'\to {\bf G}(k)$, $\rho(\Gamma\cap\Lambda')$ is
bounded in ${\bf G}(k)$.
\end{enumerate}
Furthermore, if $\Lambda$ is assumed to be finitely generated, the class of fields considered in (2) could be taken to be the class of local fields.
\end{theorem}

Taking $\Gamma=\Lambda$ we get the following corollary,
which is an extension of the linearity criterion \cite[Theorem~A.1]{BCL}, not adhering to a finite generation assumption.

\begin{cor} \label{cor:nonlinearity}
For any countable group $\Gamma$ the following properties are equivalent.
\begin{enumerate}

\item For every field $K$, integer $d$ and a group homomorphism $\phi:\Gamma \to \GL_d(K)$, $\phi(\Gamma)$ is solvable-by-locally finite.

\item For every finite index subgroup $\Gamma'<\Gamma$, complete field with an absolute value $k$, connected adjoint $k$-simple algebraic group ${\bf G}$ and Zariski dense group homomorphism $\rho:\Gamma'\to {\bf G}(k)$, $\rho(\Gamma')$ is
bounded in ${\bf G}(k)$.
\end{enumerate}
Furthermore, if $\Gamma$ is assumed to be finitely generated, the class of fields considered in (2) could be taken to be the class of local fields.
\end{cor}

Before proving Theorem~\ref{lem:nonlinearity} let us make some preliminary lemmas.
The following important lemma is essentially due to Breuillard and Gelander,
and it is an elaboration of an older lemma of Tits.

\begin{lemma} \label{lem:BG}
Let $K$ be a countable field, $R<K$ a finitely generated subring and 
$I\subset R$ an infinite subset.
Then there exist a complete field with an absolute value $k$ 
and an embedding $K\hookrightarrow k$ such that the image
of $I$ in $k$ is unbounded.
Furthermore, if $K$ is finitely generated as a field then $k$ could
be taken to be a local field.
\end{lemma}

\begin{proof}
The case where $K$ is the fraction field of $R$ is dealt with in
\cite[Lemma 2.1]{BreuillardGelander-Tits}.
If $K$ is finitely generated, say by a finite set $A$,
upon replacing $R$ with the ring $R'$ generated by $A\cup B$ where
$B$ is a finite set of generators of $R$, we are reduced to the previous
case.
The general case follows from the case where $K$ is finitely generated as follows.
Denote by $K_0$ the fraction field of $R$ and observe that it is finitely 
generated as a field.
Let $k_0$ be a corresponding local field such that
there exists an embedding $i:K_0\hookrightarrow k_0$ with $i(I)$ unbounded. 
Let $\bar{k}_0$ be an algebraic closure of $k_0$ and note that the absolute value of $k_0$ extends uniquely to
an absolute value on $\bar{k}_0$ \cite{BGR}*{\S 3.2.4, Theorem 2}.
Let $k$ be the completion of $\bar{k}_0$ with respect to this absolute value.
Note that $k$ is an algebraically closed field \cite{BGR}*{\S 3.4.1, Proposition 3} which is of uncountable transcendental degree over its prime field.
Denote by $j$ the composition $K_0\hookrightarrow k_0\hookrightarrow k$
and note that $j(I)$ is unbounded.
As $K_0$ is countable, $k$ is of uncountable transcendental degree over $j(K_0)$.
As $K$ is countable and $k$ is algebraically closed and of uncountable transcendental degree over $j(K_0)$, $j$ extends to an the embedding $K \hookrightarrow k$.
\end{proof}

\begin{lemma} \label{lem:polishfield}
Let $K$ be a field and ${\bf G}$ a $K$-algebraic group.
Let $\Gamma$ be a subgroup of ${\bf G}(K)$ which is not locally finite.
Then there exists a complete field with an absolute value $k$ 
and an embedding $K\hookrightarrow k$ such that the image of $\Gamma$ is unbounded in ${\bf G}(k)$. 
Furthermore, if the field $K$ is finitely generated then the field $k$ could be taken to be a local field.
\end{lemma}

\begin{proof}
We let $\Gamma'<\Gamma$ be an infinite finitely generated group.
We fix an injective $K$-representation ${\bf G}\to \GL_n$, let $I\subset K$
be the set of matrix coefficients of $\Gamma'$ and let $R<K$ be the subring
generated by $I$. Note that $I$ is infinite and $R$ is finitely generated.
We let $K\hookrightarrow k$ be a field extension as provided by Lemma~\ref{lem:BG} and note that the image of $\Gamma$ is unbounded
in ${\bf G}(k)$, as the image of $\Gamma'$ is unbounded in $\GL_n(k)$.
\end{proof}

% We also need the following consequence of the Tits alternative.
% 
% \begin{theorem}[Tits \cite{Tits}*{Theorem 2}] \label{T:amenK}
% Let $K$ be a field and $\Gamma$ be a subgroup
% of $\GL_n(K)$ for some integer $n$.
% Then $\Gamma$ is amenable iff $\Gamma$ is solvable-by-locally finite.
% \end{theorem}

% \begin{proof}
% This lemma is known to hold in case $\Gamma$ is finitely generated.
% The general case follows easily, using the fact that there is a uniform bound on the solvability degree of solvable subgroups in $\GL_n(K)$.
% \end{proof}

%The following proposition will not be used in the sequel, but we state %it 
%anyway, for the completeness of our discussion.
%The proof will not be given, but we note that it is an easy application %of
%Jordan's theorem giving a bound on the index of 
%a normal abelian subgroups of linear groups.
%
%\begin{prop}
%Let $K$ be a field and $n$ an integer .
%There exists a number $m$ such that for every subgroup and %$\Gamma<\GL_n(K)$,
%if $\Gamma$ is solvable-by-locally finite then there exists a solvable %normal subgroup
%$S\lhd \Gamma$ such that for every $\gamma\in\Gamma/S$, 
%$\gamma^m=e$.
%\end{prop}

\begin{lemma} \label{lem:Zdense}
Fix a field $K$ and a connected adjoint $K$-simple algebraic group ${\bf G}$.
Let $\Lambda<{\bf G}(K)$ be a Zariski dense subgroup.
Let $\Gamma<\Lambda$ a commensurated subgroup.
Assume $\Gamma$ is infinite.
Then $\Gamma$ is Zariski dense in ${\bf G}$.
\end{lemma}

\begin{proof}
Consider the Zariski closure ${\bf H}=\overline{\Gamma}^Z$
and let ${\bf H}^0$ be its identity connected component.
By \cite[AG, Theorem~14.4]{borel} $\mathbf{H}$  is defined over $K$ and $\Gamma<\mathbf{H}(K)$.
We set $\Gamma'=\Gamma \cap {\bf H}^0(K))$.
As ${\bf H}^0(K)<{\bf H}(K)$ is of finite index, $\Gamma'<\Gamma$
is of finite index.
It follows that $\Gamma'<\Lambda$ is commensurated.
Thus ${\bf H}^0<{\bf G}$ is commensurated.
As the connected group ${\bf H}^0$ has no Zariski closed subgroups of finite index, 
we conclude that ${\bf H}^0$ is normal in ${\bf G}$.
By the assumption that $\Gamma$ is infinite we get that ${\bf H}^0$ is
non-trivial.
By the simplicity of ${\bf G}$ we conclude that ${\bf H}^0={\bf G}$.
In particular ${\bf H}={\bf G}$.
Thus indeed, $\Gamma$ is Zariski dense in ${\bf G}$.
\end{proof}

\begin{lemma} \label{lem:amen}
Let $k$ be a field with absolute value and ${\bf G}$ a semisimple algebraic group defined over $k$.
Let $R$ be a topological group which is separable and amenable\footnote{Recall that through out this paper we use the convention that a topological group is amenable
if every continuous action of it on a compact space admits an invariant measure
(note that various notions of amenability which are equivalent for lcsc groups need not coincide in general).}
and let $\rho:R \to \bfG(k)$ be a continuous homomorphism with Zariski dense image.
Then $\rho(R)$ is bounded in $\bfG(k)$.
\end{lemma}

This lemma is essentially claimed in \cite[Corollary~1.18]{BDL} and it indeed follows
from the considerations of \cite{BDL}. 
However, \cite[Corollary~1.18]{BDL} is not a formal corollary of  \cite[Corollary~1.17]{BDL} as claimed in \cite{BDL},
as the latter considers merely locally compact groups.
We thank the anonymous referee for spotting this gap,
which is closed by the proof below.

\begin{proof}
We assume as we may that $k$ is non-discrete, complete and separable as a metric space.
We argue as in the proof of \cite[Theorem~6.1]{BDL}.
By \cite[Lemma~4.5]{BDL} we may assume ${\bf G}$ is adjoint
and, upon replacing $R$ with the product of its projections to the various simple factors, the problem is easily reduced to the case where ${\bf G}$ is $k$-simple.
We thus assume ${\bf G}$ is indeed $k$-simple.

Using \cite[Proposition~1.10]{borel} there exists a $k$-closed immersion of ${\bf G}$ into some $\GL_n$. 
Fixing a minimum dimensional such representation, using the fact that ${\bf G}$ is $k$-simple, we assume as we may that this representation is $k$-irreducible.
By the fact that ${\bf G}$ is adjoint, the associated morphism ${\bf G}\to \PGL_n$ is a closed immersion as well.
We will denote for convenience $E=k^n$.
Via this representation, ${\bf G}(k)$  acts continuously and faithfully on the metric space of homothety classes of norms, $I(E)$,
and on the compact space of homothety classes of seminorms, $S(E)$, introduced in \cite[\S5]{BDL}.

Using the amenability of $R$ there exits an $R$-invariant ergodic probability measure
$\mu$ on $S(E)$ which we now fix.
This measure is also fixed by the closure of $\rho(R)$,
thus, upon replacing $R$ with the closure of $\rho(R)$, we assume as we may that $R$ is a closed subgroup of $\bfG(k)$.
By \cite[Proposition~5.4]{BDL}, there is a ${\bf G}(k)$-invariant measurable partition $S(E)=\bigcup_{d=0}^{n-1} S_d(E)$, 
given by the dimension of the kernels of the seminorms.
By ergodicity $\mu$ is supported on $S_d(E)$ for some $0\leq d \leq n-1$.

We assume now $d>0$.
We denote by $\Gr_d(E)$ the Grassmannian  of $d$-dimensional $k$-subspaces of $E$ and consider 
the map $S_d(E)\to \Gr_d(E)$ taking a seminorm to its kernel. This map is measurable by \cite[Proposition~5.4]{BDL}.
Pushing forward the measure $\mu$ we obtain an $R$-invariant measure $\nu$ on $\Gr_d(E)$.
By \cite[Proposition~1.9]{BDL} there exists a $k$-subgroup ${\bf H}<{\bf G}$
which is normalized by $R$ such that the image of $R$ is precompact in $N_{\bf G}({\bf H})/{\bf H}(k)$ 
and $\nu$ is supported on the set of 
${\bf H}$-fixed points in $\Gr_d(E)$. As $R<{\bf G}$ is Zariski dense and ${\bf G}$ is $k$-simple, 
we deduce that either ${\bf H}=\{e\}$ or ${\bf H}={\bf G}$. 
By the irreducibility of our chosen representation there are no 
${\bf G}$-fixed points in $\Gr(E)$ and we conclude that ${\bf H}=\{e\}$. 
It follows that $R$ is compact in ${\bf G}(k)$ and, in particular, it is bounded.

Next we assume $d=0$,
that is $\mu$ is supported on $I(E)$.
By \cite[Lemma~5.1]{BDL} $I(E)$ is a metric space on which ${\bf G}(k)$ acts continuously and isometrically and the stabilizers of bounded subsets of $I(E)$ are bounded in ${\bf G}(k)$.
We find a ball $B\subset I(E)$ such that $\mu(B)>1/2$.
It follows that for any $g\in R$, $gB$ intersects $B$.
Thus the set $RB$ is bounded in $I(E)$.
It follows that the stabilizer in ${\bf G}(k)$ of the set $RB$ is bounded.
As $R$ is contained in this stabilizer, we conclude that $R$ is indeed bounded.
\end{proof}

%%%%%%%%%%%%%%

\begin{proof}[Proof of {Theorem~\ref{lem:nonlinearity}}]
We first observe the easier implication $(1)\Rightarrow (2)$.
We assume by contradiction having a finite index subgroup $\Lambda'<\Lambda$, a complete field with absolute value $k$, a connected adjoint $k$-simple algebraic group ${\bf G}$ and group homomorphism $\rho:\Lambda'\to {\bf G}(k)$ such that $\rho(\Lambda')$ is
Zariski dense in ${\bf G}$ and $\rho(\Gamma\cap\Lambda')$ is unbounded in ${\bf G}(k)$.
We set $\Gamma'=\Gamma\cap\Lambda'$ and observe that
$\Gamma'<\Gamma$ is of finite index.
Note that $\rho(\Gamma')$ is unbounded in ${\bf G}(k)$,
hence infinite, and it is
commensurated by the Zariski dense subgroup $\rho(\Lambda')$.
By Lemma~\ref{lem:Zdense} we get that $\rho(\Gamma')$ is
Zariski dense in ${\bf G}$.
By Lemma~\ref{lem:amen} we conclude that $\rho(\Gamma')$ is not
amenable.
By embedding ${\bf G}(k)$ in $\GL_{d}(k)$
we consider $\rho$ as a linear representation $\rho:\Lambda'\to \GL_{d}(k)$.
We induce $\rho$ from $\Lambda'$ to $\Lambda$ and obtain
a representation $\phi:\Lambda\to \GL_{d'}(k)$.
$\phi(\Gamma)$ is not amenable as it contains the non-amenable subgroup $\phi(\Gamma')\simeq \rho(\Gamma')$.
It follows that $\phi(\Gamma)$ is not solvable-by-locally finite.
This completes the proof of $(1)\Rightarrow (2)$

We now focus on the implication $(2)\Rightarrow (1)$.
We assume by contradiction having a field $K$, an integer $d$ and a group 
homomorphism $\phi:\Lambda \to \GL_d(K)$ such that $\phi(\Gamma)$ is not solvable-by-locally finite.
We assume as we may that $K$ is a countable field, and in case that 
$\Lambda$ is finitely generated we assume also that $K$ is finitely generated.
Tits' Theorem \cite{Tits}*{Theorems 1 and 2} implies that $\phi(\Gamma)$ is not amenable.
We let $\mathbf{H}$ be the Zariski closure of $\phi(\Lambda)$. 
By \cite[AG, Theorem~14.4]{borel} $\mathbf{H}$  is defined over $K$ and $\phi(\Lambda)<\mathbf{H}(K)$.
We denote by $\mathbf{H}^0$ the identity connected component in $\mathbf{H}$
and observe that $\mathbf{H}^0(K)<\mathbf{H}(K)$ is of finite index.
We set $\Lambda'=\phi^{-1}(\mathbf{H}^0(K))$.
Note that $\Lambda'<\Lambda$ is of finite index and the Zariski closure of $\phi(\Lambda')$ is $\mathbf{H}^0$, 
as it is contained in the connected group $\mathbf{H}^0$ and it is of finite index in $\mathbf{H}$.

We set $\Gamma'=\Lambda'\cap\Gamma$.
Observe that $\Gamma'<\Gamma$ is of finite index.
In particular, $\Gamma'<\Lambda'$ is commensurated and $\phi(\Gamma')$
is not amenable.
We let $\mathbf{R}$ be the solvable radical of $\mathbf{H}^0$ and set $\mathbf{L}=\mathbf{H}^0/\mathbf{R}$, $\pi:\mathbf{H}^0\to \mathbf{L}$.
As $\mathbf{R}(K)$ is solvable and $\phi(\Gamma')$ is not amenable, we conclude that $\pi\circ\phi(\Gamma')$ is not amenable.
In particular, $\mathbf{L}$ is non-trivial. Upon replacing $\mathbf L$ by its quotient modulo its center, which is finite, we may further assume that $\mathbf L$ is center-free.
Thus $\mathbf{L}$ is a connected, adjoint, semisimple $K$-group and $\pi\circ\phi(\Gamma')$ is Zariski dense in $\mathbf{L}$.
The adjoint group $\mathbf{L}$ is $K$-isomorphic to a direct product of finitely many connected, adjoint, $K$-simple $K$-groups.
The image of the projection of $\pi\circ\phi(\Gamma')$ to at least one of these factors is not amenable, and in particular not locally finite.
We fix such a $K$-simple factors of $\mathbf{L}$ and denote it
$\mathbf{G}$.
We let $\psi \colon \Lambda'\to \mathbf{G}(K)$ be the composition of $\pi\circ \phi$ with the projection $\mathbf{L}(K)\to \mathbf{G}(K)$
and conclude that $\psi(\Lambda')$ is Zariski dense and $\psi(\Gamma')$ is not locally finite in the connected adjoint $K$-simple $K$-algebraic group $\mathbf{G}$.
Using Lemma~\ref{lem:polishfield} we find a complete field with an absolute value $k$ 
and an embedding $K\hookrightarrow k$ such that the image of $\Gamma'$ is unbounded in ${\bf G}(k)$,
which is furthermore a local field if the field $K$ is finitely generated. 
We obtain a contradiction, as the image of $\Gamma'$ is unbounded under the composed map $\Lambda' \to {\bf G}(K)\to {\bf G}(k)$.
\end{proof}

\section{Proofs of Theorem~\ref{thm:nonlinearity}, and Corollaries~\ref{cor:nonlinearitycom}, \ref{cor:newamenable}} \label{sec:proffnl}

\begin{proof}[Proof of Theorem~\ref{thm:nonlinearity}]
Assume by contradiction that there exist an integer $d$, a field $K$ and a linear representation $\phi:\Gamma \to \GL_d(K)$ 
such that the image $\phi(\Gamma)$ is not solvable-by-locally finite.
Then, by Corollary~\ref{cor:nonlinearity},
there exist: a finite index subgroup $\Gamma'<\Gamma$, a complete field with an absolute value $k$, 
a connected adjoint $k$-simple algebraic group ${\bf G}$ and group homomorphism $\rho:\Gamma'\to {\bf G}(k)$ 
such that $\rho(\Gamma')$ is Zariski dense and unbounded in ${\bf G}(k)$.
In case $\Gamma$ is finitely generated $k$ could be taken to be a local field.

For every $i=1,\ldots,n$ let $T'_i=\overline{\pi_i(\Gamma')}<T_i$ be the closure of the projection of $\Gamma'$ in $T_i$,
and $T'=T'_1\times\cdots \times T'_n<T$ be a closed subgroup.
Since $\Gamma'<\Gamma$ has finite index, each $T'_i<T_i$ have finite index, and so $T'<T$ has finite index.
Note also that $\Gamma'<T'$ is a lattice with dense projections.
Applying Theorem~\ref{thm:lattice} we conclude that there exists a continuous homomorphism $\bar{\rho}:T' \to {\bf G}(k)$ 
such that $\bar{\rho}|_{\Gamma'}=\rho$. 
In particular, $\bar{\rho}(T')$ is Zariski dense and unbounded and so is its closure, $\overline{\bar{\rho}(T')}$.
By Lemma~\ref{lem:amen} we conclude that $\overline{\bar{\rho}(T')}$ is not amenable.

Fixing a $k$-representation ${\bf G}(k)\to \GL_{d'}(k)$ we get a continuous linear representation $T'\to \GL_{d'}(k)$.
Inducing this representation to $T$, we get a
continuous homomorphism $\psi:T\to \GL_{d''}(k)$ such that
$\overline{\psi(T)}$ is not amenable.
This is a contradiction.
\end{proof}

\begin{proof}[Proof of Corollary~\ref{cor:nonlinearitycom}]
Assume by contradiction that there exist an integer $d$, a field $K$ and a linear representation 
$\phi:\Lambda \to \GL_d(K)$ such that $\phi(\Gamma)$ is not solvable-by-locally finite.
By Theorem~\ref{lem:nonlinearity} we get
a finite index subgroup $\Lambda'<\Lambda$, a complete field with an absolute value $k$, 
a connected adjoint $k$-simple algebraic group ${\bf G}$ and a Zariski dense group homomorphism 
$\rho:\Lambda'\to {\bf G}(k)$ such that $\rho(\Gamma\cap\Lambda')$ is unbounded in ${\bf G}(k)$.
In case $\Lambda$ is finitely generated $k$ could be taken to be a local field.

We let $T'$ be the closure of $\Lambda'$ in $T$ and set $\Gamma'=
\Gamma\cap \Lambda'$.
Observe that $T'<T$ is a closed subgroup of finite index.
Thus $T'<T$ is open and $\Gamma'<T'$ is a lattice.
Furthermore, $\Lambda'$ is a countable dense subgroup of $T'$ containing and commensurating $\Gamma'$.
Applying Corollary~\ref{cor:com} we get
a continuous homomorphism $\bar{\rho}:T'\to {\bf G}(k)$
such that $\rho=\bar{\rho}|_{\Lambda'}$.
As $\rho(\Gamma')$ is unbounded we get that $\bar{\rho}(T')$ is unbounded.
By Lemma~\ref{lem:amen} we conclude that $\overline{\bar{\rho}(T')}$ is not amenable.

Fixing a $k$-representation ${\bf G}(k)\to \GL_{d'}(k)$ we get a continuous linear representation $T'\to \GL_{d'}(k)$.
Inducing this representation to $T$, we get a
continuous homomorphism $\psi:T\to \GL_{d''}(k)$ such that
$\overline{\psi(T)}$ is not amenable.
This is a contradiction.
\end{proof}

\begin{proof}[Proof of Corollary~\ref{cor:newamenable}]
Assume by contradiction that there exist an integer $d$, a field $K$ and a linear representation $\phi:\Gamma \to \GL_d(K)$ 
such that $\phi(\Gamma)$ is not solvable-by-locally finite.
In particular, by Tits' Theorem \cite{Tits}*{Theorems 1 and 2}, $\phi(\Gamma)$ is not amenable.
We claim that  $\phi$ is injective.
Assume by contradiction that $N=\ker \phi$ is non-trivial.
By the assumption 
that for each $i$, $\pi_i|\Gamma$ is injective and for each non-trivial closed normal subgroup of $T_i$
the corresponding quotient group is amenable, we get that $T_i/\overline{\pi_i(N)}$ is amenable.
This contradicts \cite[Theorem~3.7(iv)]{BS}, which guarantees that then $\phi(\Gamma)\simeq \Gamma/N$ is amenable.
We conclude that, indeed, $\phi$ is injective.

We set $S=T_2\times \cdots \times T_n$, let $U<S$ be a compact open subgroup and set $\Delta=\Gamma
\cap (T_1\times U)$. Note that $\Delta$ is a lattice in $T_1\times U$, as $\Gamma$ is a lattice in $T$ and
$T_1\times U$ is open in $T$.
As $U$ is compact, we deduce that $\pi_1(\Delta)<T_1$ is a lattice.
Since $U<S$ is commensurated, we get that $T_1\times U<T=T_1\times S$ is commensurated, thus $\Delta<\Gamma$ is commensurated.

As $\pi_1$ is injective on $\Gamma$, we may identify $\Gamma$ and
$\Delta$ with $\pi_1(\Gamma)$ and $\pi_1(\Delta)$ correspondingly, and 
consider them as subgroups of $T_1$.
Applying Corollary~\ref{cor:nonlinearitycom} in this setting,
we get that $\phi(\Delta)$ is solvable-by-locally finite, and in particular it is amenable.
By the injectivity of $\phi$ we conclude that $\Delta$ is amenable.
As $\Delta<T_1$ is a lattice, it follows that $T_1$ is amenable.
This is a contradiction.
\end{proof}

%%%%%%%%%%%%%%%


\begin{bibdiv}
\begin{biblist}

	\bib{BCL}{article}{
		   author={Bader, Uri},
		   author={Caprace, Pierre-Emmanuel},
		   author={L\'{e}cureux, Jean},
		   title={On the linearity of lattices in affine buildings and ergodicity of
		   the singular Cartan flow},
		   journal={J. Amer. Math. Soc.},
		   volume={32},
		   date={2019},
		   number={2},
		   pages={491--562},
		   issn={0894-0347},
		   review={\MR{3904159}},
		   doi={10.1090/jams/914},
	}


	\bib{BDL}{article}{
		AUTHOR = {Bader, Uri}, 
		Author={Duchesne, Bruno}, 
		Author={L\'{e}cureux, Jean},
		TITLE = {Almost algebraic actions of algebraic groups and
			applications to algebraic representations},
	  JOURNAL = {Groups Geom. Dyn.},
	  FJOURNAL = {Groups, Geometry, and Dynamics},
	    VOLUME = {11},
	      YEAR = {2017},
	    NUMBER = {2},
	     PAGES = {705--738},
	}

		\bib{BF:Margulis}{article}{
		AUTHOR = {Bader, Uri}, 
		Author={Furman, Alex},
		TITLE = {An extension of Margulis' Super-Rigidity Theorem},
		 YEAR = {2018},
		eprint={https://arxiv.org/abs/1810.01607},
	}
	
	\bib{AREA}{article}{
		AUTHOR = {Bader, Uri}, 
		Author={Furman, Alex},
		TITLE = {Algebraic Representations of Ergodic Actions and Super-Rigidity},
 		YEAR = {2014},
		pages={1--25},
		eprint={https://arxiv.org/abs/1311.3696},
	}

\bib{BS}{article}{
    AUTHOR = {Bader, Uri}, 
    Author ={Shalom, Yehuda},
     TITLE = {Factor and normal subgroup theorems for lattices in products
              of groups},
   JOURNAL = {Invent. Math.},
  FJOURNAL = {Inventiones Mathematicae},
    VOLUME = {163},
      YEAR = {2006},
    NUMBER = {2},
     PAGES = {415--454},
      ISSN = {0020-9910},
   MRCLASS = {22E40 (28D15 43A05 43A07)},
  MRNUMBER = {2207022},
MRREVIEWER = {Alain Valette},
       URL = {https://doi.org/10.1007/s00222-005-0469-5},
}

\bib{BF-lyap}{article}{
   author={Bader, Uri},
   author={Furman, Alex},
   title={Boundaries, rigidity of representations, and Lyapunov exponents},
   conference={
      title={Proceedings of the International Congress of
      Mathematicians---Seoul 2014. Vol. III},
   },
   book={
      publisher={Kyung Moon Sa, Seoul},
   },
   date={2014},
   pages={71--96},
   review={\MR{3729019}},
}

\bib{BFS}{article}{
 	AUTHOR = {Bader, Uri}, 
 	Author={Furman, Alex},
 	Author={Sauer, Roman}
 	TITLE = {Lattice envelopes},
 	JOURNAL = {Duke Math. J.},
 	YEAR = {2019},
	note={to appear},
 	eprint={https://arXiv:1711.08410 },
%PAGES={71--96},
}


\bib{b-z}{article}{
    AUTHOR = {Bern{\v{s}}te{\u\i}n, I. N.},
AUTHOR={Zelevinski{\u\i}, A. V.},
     TITLE = {Representations of the group {$GL(n,F),$} where {$F$} is a
              local non-{A}rchimedean field},
   JOURNAL = {Uspehi Mat. Nauk},
  FJOURNAL = {Akademiya Nauk SSSR i Moskovskoe Matematicheskoe Obshchestvo.
              Uspekhi Matematicheskikh Nauk},
    VOLUME = {31},
      YEAR = {1976},
    NUMBER = {3(189)},
     PAGES = {5--70},
      ISSN = {0042-1316},
   MRCLASS = {22E50},
  MRNUMBER = {0425030 (54 \#12988)},
MRREVIEWER = {G. I. Olsanskii},
}

\bib{borel}{book}{
    AUTHOR = {Borel, Armand},
     TITLE = {Linear algebraic groups},
    SERIES = {Graduate Texts in Mathematics},
    VOLUME = {126},
   EDITION = {Second},
 PUBLISHER = {Springer-Verlag},
   ADDRESS = {New York},
      YEAR = {1991},
     PAGES = {xii+288},
      ISBN = {0-387-97370-2},
   MRCLASS = {20-01 (20Gxx)},
  MRNUMBER = {1102012 (92d:20001)},
MRREVIEWER = {F. D. Veldkamp},
       DOI = {10.1007/978-1-4612-0941-6},
       URL = {http://dx.doi.org/10.1007/978-1-4612-0941-6},
}

\bib{BGR}{book}{
   author={Bosch, S.},
   author={G\"{u}ntzer, U.},
   author={Remmert, R.},
   title={Non-Archimedean analysis},
   series={Grundlehren der Mathematischen Wissenschaften [Fundamental
   Principles of Mathematical Sciences]},
   volume={261},
   note={A systematic approach to rigid analytic geometry},
   publisher={Springer-Verlag, Berlin},
   date={1984},
   pages={xii+436},
   isbn={3-540-12546-9},
   review={\MR{746961}},
   doi={10.1007/978-3-642-52229-1},
}
\bib{BreuillardGelander-Tits}{article}{
      author={Breuillard, E.},
      author={Gelander, T.},
       title={A topological {T}its alternative},
        date={2007},
     journal={Ann. of Math. (2)},
      volume={166},
      number={2},
       pages={427\ndash 474},
}

\bib{Bruhat+Tits}{article}{
   author={Bruhat, F.},
   author={Tits, J.},
   title={Groupes r\'{e}ductifs sur un corps local},
   language={French},
   journal={Inst. Hautes \'{E}tudes Sci. Publ. Math.},
   number={41},
   date={1972},
   pages={5--251},
   issn={0073-8301},
   review={\MR{0327923}},
}

\bib{BM}{article}{
    AUTHOR = {Burger, M.}, Author={Monod, N.},
     TITLE = {Continuous bounded cohomology and applications to rigidity
              theory},
   JOURNAL = {Geom. Funct. Anal.},
  FJOURNAL = {Geometric and Functional Analysis},
    VOLUME = {12},
      YEAR = {2002},
    NUMBER = {2},
     PAGES = {219--280},
      ISSN = {1016-443X},
     CODEN = {GFANFB},
   MRCLASS = {53C24 (22E41 46H25 46M20)},
  MRNUMBER = {1911660 (2003d:53065a)},
MRREVIEWER = {David Michael Fisher},
       DOI = {10.1007/s00039-002-8245-9},
       URL = {http://dx.doi.org/10.1007/s00039-002-8245-9},
}


\bib{CM09}{article}{
    AUTHOR = {Caprace, Pierre-Emmanuel}, Author={Monod, Nicolas}
     TITLE = {Isometry groups of non-positively curved spaces:
{D}iscrete subgroups},
   JOURNAL = {Journal of Topology},
    VOLUME = {2},
  NUMBER = {4},
      YEAR = {2009},
     PAGES = {701--746},
}






\bib{effros}{article}{
    AUTHOR = {Effros, Edward G.},
     TITLE = {Transformation groups and {$C^{\ast} $}-algebras},
   JOURNAL = {Ann. of Math. (2)},
  FJOURNAL = {Annals of Mathematics. Second Series},
    VOLUME = {81},
      YEAR = {1965},
     PAGES = {38--55},
      ISSN = {0003-486X},
   MRCLASS = {46.65},
  MRNUMBER = {0174987 (30 \#5175)},
MRREVIEWER = {J. M. G. Fell},
}


\bib{valued}{book}{
    AUTHOR = {Engler, Antonio J.}, Author={Prestel, Alexander},
     TITLE = {Valued fields},
    SERIES = {Springer Monographs in Mathematics},
 PUBLISHER = {Springer-Verlag},
   ADDRESS = {Berlin},
      YEAR = {2005},
     PAGES = {x+205},
      ISBN = {978-3-540-24221-5; 3-540-24221-X},
   MRCLASS = {12J20 (12F05 12J10 12J12 12J15)},
  MRNUMBER = {2183496 (2007a:12005)},
MRREVIEWER = {Niels Schwartz},
}







\bib{GKM}{article}{
    AUTHOR = {Gelander, Tsachik}, Author={Karlsson, Anders}, Author={Margulis, Gregory A.},
     TITLE = {Superrigidity, generalized harmonic maps and uniformly convex
              spaces},
   JOURNAL = {Geom. Funct. Anal.},
  FJOURNAL = {Geometric and Functional Analysis},
    VOLUME = {17},
      YEAR = {2008},
    NUMBER = {5},
     PAGES = {1524--1550},
      ISSN = {1016-443X},
     CODEN = {GFANFB},
   MRCLASS = {53C24 (22D12 22E40 58E20)},
  MRNUMBER = {2377496 (2009a:53074)},
MRREVIEWER = {Raul Quiroga-Barranco},
       DOI = {10.1007/s00039-007-0639-2},
       URL = {http://dx.doi.org/10.1007/s00039-007-0639-2},
}

\bib{GW}{article}{
    AUTHOR = {Glasner, E.}, Author={Weiss, B.},
     TITLE = {Weak mixing properties for nonsingular actions},
 JOURNAL = {preprint},
}


\bib{kaimanovich}{article}{
    AUTHOR = {Kaimanovich, V. A.},
     TITLE = {Double ergodicity of the {P}oisson boundary and applications
              to bounded cohomology},
   JOURNAL = {Geom. Funct. Anal.},
  FJOURNAL = {Geometric and Functional Analysis},
    VOLUME = {13},
      YEAR = {2003},
    NUMBER = {4},
     PAGES = {852--861},
      ISSN = {1016-443X},
     CODEN = {GFANFB},
   MRCLASS = {60G50 (28C10 28D15 37A20)},
  MRNUMBER = {2006560 (2004k:60128)},
MRREVIEWER = {Gernot Greschonig},
       DOI = {10.1007/s00039-003-0433-8},
       URL = {http://dx.doi.org/10.1007/s00039-003-0433-8},
}

\bib{kechris}{book}{
    AUTHOR = {Kechris, Alexander S.},
     TITLE = {Classical descriptive set theory},
    SERIES = {Graduate Texts in Mathematics},
    VOLUME = {156},
 PUBLISHER = {Springer-Verlag},
   ADDRESS = {New York},
      YEAR = {1995},
     PAGES = {xviii+402},
      ISBN = {0-387-94374-9},
   MRCLASS = {03E15 (03-01 03-02 04A15 28A05 54H05 90D44)},
  MRNUMBER = {1321597 (96e:03057)},
MRREVIEWER = {Jakub Jasi{\'n}ski},
       DOI = {10.1007/978-1-4612-4190-4},
       URL = {http://dx.doi.org/10.1007/978-1-4612-4190-4},
}


\bib{Lifschitz}{article}{
    AUTHOR = {Lifschitz, L.},
     TITLE = {Arithmeticity of rank-1 lattices with dense commensurators in positive characteristic},
   JOURNAL = {J. Algebra},
    VOLUME = {261},
      YEAR = {2003},
     PAGES = {44--53},
}



\bib{margulis-book}{book}{
    AUTHOR = {Margulis, G. A.},
     TITLE = {Discrete subgroups of semisimple {L}ie groups},
    SERIES = {Ergebnisse der Mathematik und ihrer Grenzgebiete (3) [Results
              in Mathematics and Related Areas (3)]},
    VOLUME = {17},
 PUBLISHER = {Springer-Verlag},
   ADDRESS = {Berlin},
      YEAR = {1991},
     PAGES = {x+388},
      ISBN = {3-540-12179-X},
   MRCLASS = {22E40 (20Hxx 22-02 22D40)},
  MRNUMBER = {1090825 (92h:22021)},
MRREVIEWER = {Gopal Prasad},
}


		


\bib{Monod-products}{article}{
   AUTHOR = {Monod, Nicolas},
     TITLE = {Superrigidity for irreducible lattices and geometric
              splitting},
   JOURNAL = {J. Amer. Math. Soc.},
  FJOURNAL = {Journal of the American Mathematical Society},
    VOLUME = {19},
      YEAR = {2006},
    NUMBER = {4},
     PAGES = {781--814},
      ISSN = {0894-0347},
   MRCLASS = {22F05 (20F65 22E40 53C24)},
  MRNUMBER = {2219304 (2007b:22025)},
MRREVIEWER = {David Michael Fisher},
       DOI = {10.1090/S0894-0347-06-00525-X},
       URL = {http://dx.doi.org/10.1090/S0894-0347-06-00525-X},
}

\bib{Monod}{article}{
    AUTHOR = {Monod, Nicolas},
     TITLE = {Arithmeticity vs. nonlinearity for irreducible lattices},
   JOURNAL = {Geom. Dedicata},
  FJOURNAL = {Geometriae Dedicata},
    VOLUME = {112},
      YEAR = {2005},
     PAGES = {225--237},
      ISSN = {0046-5755},
   MRCLASS = {22E40 (20G30)},
  MRNUMBER = {2163901},
MRREVIEWER = {Dave Witte Morris},
       URL = {https://doi.org/10.1007/s10711-004-6162-9},
}

\bib{Raghunathan}{article}{
   AUTHOR = {M.S. Raghunathan},
     TITLE = {Discrete subgroups of algebraic groups over local fields of positive characteristic},
   JOURNAL = {Proc.
Indian Acad. Sci},
    VOLUME = {99},
      YEAR = {1989},
     PAGES = {127--146},
}



\bib{Rosendal}{article}{
      author={Rosendal, Christian},
       title={Automatic continuity of group homomorphisms},
        date={2009},
     journal={Bull. Symbolic Logic},
      volume={15},
      number={2},
       pages={184-- 214},
}

	
\bib{serre}{book}{
    AUTHOR = {Serre, Jean-Pierre},
     TITLE = {Lie algebras and {L}ie groups},
    SERIES = {Lecture Notes in Mathematics},
    VOLUME = {1500},
      NOTE = {1964 lectures given at Harvard University,
              Corrected fifth printing of the second (1992) edition},
 PUBLISHER = {Springer-Verlag},
   ADDRESS = {Berlin},
      YEAR = {2006},
     PAGES = {viii+168},
      ISBN = {978-3-540-55008-2; 3-540-55008-9},
   MRCLASS = {17-01 (22-01)},
  MRNUMBER = {2179691 (2006e:17001)},
}


\bib{Tits}{article}{
   author={Tits, J.},
   title={Free subgroups in linear groups},
   journal={J. Algebra},
   volume={20},
   date={1972},
   pages={250--270},
   issn={0021-8693},
   review={\MR{0286898}},
   doi={10.1016/0021-8693(72)90058-0},
}
		





\bib{zimmer-book}{book}{
   author={Zimmer, R. J.},
   title={Ergodic theory and semisimple groups},
   series={Monographs in Mathematics},
   volume={81},
   publisher={Birkh\"auser Verlag},
   place={Basel},
   date={1984},
   pages={x+209},
   isbn={3-7643-3184-4},
   review={\MR{776417 (86j:22014)}},
}

\end{biblist}
\end{bibdiv}
\end{document}